\documentclass{amsart}

\usepackage{color,graphicx,amsmath,amsfonts,amssymb,amsxtra,setspace,xspace}
\usepackage[colorlinks=true]{hyperref}
\hypersetup{urlcolor=blue, citecolor=red, linkcolor=blue}

\newtheorem{thm}{Theorem}[section]
\newtheorem{prop}[thm]{Proposition}
\newtheorem{lem}[thm]{Lemma}
\newtheorem{cor}[thm]{Corollary}
\newtheorem{rem}{Remark}
\numberwithin{equation}{section}

\newcommand{\R}{{\mathbb R}}
\renewcommand{\S}{{\mathbb S}}
\newcommand{\N}{{\mathbb N}}
\renewcommand{\L}{{\mathcal L\,}}
\newcommand{\Q}{{\mathcal Q}}
\newcommand{\D}{{\mathsf D}}
\newcommand{\be}[1]{\begin{equation}\label{#1}}
\newcommand{\ee}{\end{equation}}
\renewcommand{\(}{\left(}
\renewcommand{\)}{\right)}
\newcommand{\irdmu}[1]{\int_{\R^d}#1\,d\mu}
\newcommand{\irdmugen}[1]{\int_{(0,\infty)\times\M}#1\,d\mu}
\newcommand{\irdmugenn}[1]{\int#1\,d\mu}

\newcommand{\nrmr}[2]{\|#1\|_{\mathrm L^{#2}(\R)}}

\newcommand{\ird}[1]{\int_{\R^d}#1\,dx}
\newcommand{\irdsph}[2]{\int_0^\infty\int_{\S^{d-1}}#1\;r^{\kern1pt #2}\,\frac{dr}{r}\;d\omega}

\newcommand{\nrmRd}[2]{\|{#1}\|_{\mathrm L^{#2}(\R^d)}}

\newcommand{\isph}[1]{\int_{\S^{d-1}}#1\;d\omega}
\newcommand{\Lap}{\Delta}
\newcommand{\fv}{f}
\newcommand{\nrmcnd}[2]{\|{#1}\|_{\mathrm L^{#2}(\mathcal C_1)}}
\newcommand{\nrmcndgen}[2]{\|{#1}\|_{\mathrm L^{#2}(\mathcal C)}}

\newcommand{\iCgen}[1]{\int_{\mathcal C}{#1}\,d\nu}
\newcommand{\M}{\mathfrak M}
\newcommand{\iM}[1]{\int_{\M}{#1}\,d\kern1pt v_g}

\newcommand{\kk}{\mathsf k}
\newcommand{\nrml}[2]{\|{#1}\|_{\mathrm L^{#2}(\R)}}

\begin{document}

\title[Flows on cylinders, symmetry and rigidity]{Rigidity versus symmetry breaking via nonlinear flows on cylinders and Euclidean spaces}

\author{Jean Dolbeault}
\address{\hspace*{-12pt}Jean Dolbeault: Ceremade (UMR CNRS no. 7534), Universit\'e Paris-Dauphine, Place de Lattre de Tassigny, 75775 Paris 16, France}
\email{dolbeaul@ceremade.dauphine.fr}

\author{Maria J.~Esteban}
\address{\hspace*{-12pt}Maria J.~Esteban: Ceremade (UMR CNRS no. 7534), Universit\'e Paris-Dauphine, Place de Lattre de Tassigny, 75775 Paris 16, France}
\email{esteban@ceremade.dauphine.fr}

\author{Michael Loss}
\address{\hspace*{-12pt}Michael Loss: Skiles Building, Georgia Institute of Technology, Atlanta GA 30332-0160, USA}
\email{loss@math.gatech.edu}

\date{\today}
\begin{abstract} This paper is motivated by the characterization of the optimal symmetry breaking region in Caffarelli-Kohn-Nirenberg inequalities. As a consequence, optimal functions and sharp constants are computed in the symmetry region. The result solves a longstanding conjecture on the optimal symmetry range.\par
As a byproduct of our method we obtain sharp estimates for the principal eigenvalue of Schr\"odinger operators on some non-flat non-compact manifolds, which to the best of our knowledge are new.
\par
The method relies on generalized entropy functionals for nonlinear diffusion equations. It opens a new area of research for approaches related to \emph{carr\'e du champ} methods on non-compact manifolds. However key estimates depend as much on curvature properties as on purely nonlinear effects. The method is well adapted to functional inequalities involving simple weights and also applies to general cylinders. Beyond results on symmetry and symmetry breaking, and on optimal constants in functional inequalities, rigidity theorems for nonlinear elliptic equations can be deduced in rather general settings.\end{abstract}



\maketitle
\thispagestyle{empty}

\section{Introduction}\label{Sec:Intro}

Symmetry and the breaking thereof is a central theme in mathematics and the physical sciences. It is well known that symmetric energy functionals might have states of lowest energy that may or may not have these symmetries. In the latter case one says, in the language of physics, that the symmetry is \emph{broken}, \emph{i.e.}, the symmetry group of the minimizer is smaller than the symmetry group of the functional. Needless to say, for computing the optimal value of the functional it is of advantage that an optimizer be symmetric. In other contexts the breaking of symmetry leads to interesting phenomena such as crystals in which the translation invariance of a system is broken. Thus, it is of central importance to decide what symmetry types, if any, an optimizer has.

Very often functionals depend on parameters and it might be that in one parameter range the lowest energy state has the full symmetry of the functional, while in other parts of the parameter region the symmetry is broken. Thus, in each region of the parameter space the minimizers possess a fixed symmetry or, to use a term from physics, a \emph{phase} and the collection of these various phases constitute a \emph {phase diagram}.

To decide whether a minimizer has the full symmetry or not can be difficult. To show that symmetry is broken one can minimize the functional in the \emph{class of symmetric functions} and then check whether the value of the functional can be lowered by perturbing the minimizer away from the symmetric situation. If one can lower the energy in this fashion then symmetry is broken. This procedure is successful only if one knows a lot about the minimizer in the symmetric class and can sometimes be a formidable problem, but it is a \emph{local} problem. It can also happen that there is degeneracy, that is, the energy of the symmetric as well as non-symmetric minimizers are the same, \emph{i.e.}, there is a region in parameter
space where there is coexistence.

A real difficulty occurs when the minimizer in the symmetric class is stable, \emph{i.e.}, all local perturbations that break the symmetry increase the energy. It is obvious that, in general, one cannot conclude that the minimizer is symmetric because the minimizer in the symmetric class and the actual minimizer might not be close in any reasonable notion of distance. In general it is very difficult to decide, assuming stability, wether the minimizer is symmetric or not. This is a \emph{global} problem and not amenable to linear methods.

It is evident that there are no general technique available for understanding symmetry of minimizers. The focus has to be and has been on relevant and non-trivial examples, such as finding the sharp constant in Sobolev's inequality~\cite{Aubin-76,Talenti-76}, the Hardy-Littlewood-Sobolev inequality~\cite{Lieb-83} or the logarithmic Sobolev inequality~\cite{Gross75} to mention classical examples. In the former two instances, rearrangement inequalities are the main tool for establishing the symmetry of the optimizers. There is a fairly large list of such examples that make up the canon of analysis and the goal of this paper is to add another one to it namely the problem of determining the sharp constant in the \emph{Caffarelli-Kohn-Nirenberg inequalities}.

\bigskip The Caffarelli-Kohn-Nirenberg inequalities
\be{CKN}
\(\ird{\frac{|v|^p}{|x|^{b\,p}}}\)^{2/p}\le\,\mathsf C_{a,b}\ird{\frac{|\nabla v|^2}{|x|^{2\,a}}}\quad\forall\,v\in\mathcal D_{a,b}
\ee
have been established in~\cite{Caffarelli-Kohn-Nirenberg-84}, under the conditions that $a\le b\le a+1$ if $d\ge3$, $a< b\le a+1$ if $d=2$, $a+1/2< b\le a+1$ if $d=1$, and $a<a_c$ where
\[
a_c:=\frac{d-2}2\,.
\]
The exponent
\be{exponent-relationabp}
p=\frac{2\,d}{d-2+2\,(b-a)}
\ee
is determined by the invariance of the inequality under scalings. Here $\mathsf C_{a,b}$ denotes the optimal constant in~\eqref{CKN} and the space $\mathcal D_{a,b}$ is defined by
\[
\mathcal D_{a,b}:=\Big\{\,v\in\mathrm L^p\(\R^d,|x|^{-b}\,dx\)\,:\,|x|^{-a}\,|\nabla v|\in\mathrm L^2\(\R^d,dx\)\Big\}\,.
\]
The space $\mathcal D_{a,b}$ can be obtained as the completion of $C_c^\infty(\R^d)$, the space of smooth functions in $\R^d$ with compact support, with respect to the norm defined by $\|v\|^2=\|\,|x|^{-b}\,v\,\|_p^2+\|\,|x|^{-a}\,\nabla v\,\|_2^2$. Inequality~\eqref{CKN} holds also for $a>a_c$, but in this case $\mathcal D_{a,b}$ has to be defined as the completion with respect to $\|\cdot\|$ of the space $C_c^\infty(\R^d\setminus\{0\}):=\big\{w\in C_c^\infty(\R^d)\,:\,\mbox{supp}(w)\subset\R^d\setminus\{0\}\big\}$. The two cases, $a>a_c$ and $a<a_c$, are related by the property of modified inversion symmetry that can be found in~\cite[Theorem~1.4, (ii)]{Catrina-Wang-01}. In this paper, we shall assume that $a<a_c$ without further notice. Inequality~\eqref{CKN} is sometimes called the \emph{Hardy-Sobolev inequality:} for $d\ge3$ it interpolates between the usual Sobolev inequality ($a=0$, $b=0$) and the weighted Hardy inequalities corresponding to $b=a+1$. More details can be found in~\cite{Catrina-Wang-01}. In this paper, F.~Catrina and Z.-Q.~Wang have also shown existence results of optimal functions if $b>a$. For $b=a<0$, $d\ge3$, equality in~\eqref{CKN} is never achieved in $\mathcal D_{a,b}$. For $b=a+1$ and $d\ge2$, the best constant in~\eqref{CKN} is given by $\mathsf C_{a,a+1}=(a_c-a)^2$ and it is never achieved. For $a<b<a+1$ and $d\ge2$, the best constant in~\eqref{CKN} is always achieved at some {\sl extremal\/} function $v_{a,b}\in\mathcal D_{a,b}$. When $d=1$, optimal functions are even and explicit, as we shall see next.

If we consider inequality~\eqref{CKN} on the smaller set of functions in $\mathcal D_{a,b}$ which are radially symmetric, then the optimal constant is improved to a constant $\mathsf C_{a,b}^\star\ge\mathsf C_{a,b}$ and equality is achieved by
\[
v_\star(x)=\(1+|x|^{(p-2)\,(a_c-a)}\)^{-\frac2{p-2}}\quad\forall\,x\in\R^d\,.
\]
In other words, we have $\mathsf C_{a,b}^\star=\|\,|x|^{-b}\,v_\star\,\|_p^2\,/\,\|\,|x|^{-a}\,\nabla v_\star\,\|_2^2$. Moreover, all optimal radial functions are equal to $v_\star$ up to a scaling or a multiplication by a constant (and translations if $a=b=0$). If $d=1$, $\mathsf C_{a,b}=\mathsf C_{a,b}^\star$ and $v_\star$ is always an optimal function. The main \emph{symmetry issue} is to know for which value of the parameters $a$ and~$b$ the function $v_\star$ is also optimal for inequality~\eqref{CKN} when $d\ge2$ or, equivalently, for which values of $a$ and $b$ we have $\mathsf C_{a,b}=\mathsf C_{a,b}^\star$. We shall say that \emph{symmetry} holds if $\mathsf C_{a,b}=\mathsf C_{a,b}^\star$ and that we have \emph{symmetry breaking} otherwise.

\smallskip\emph{Symmetry} results have been obtained in various regions of the $(a,b)$ plane. Moving planes and symmetrization methods have been applied successfully in~\cite{Chou-Chu-93,MR1731336} and~\cite[Lemma~2.1]{DELT09} to cover the range $0\le a<a_c$ when $d\ge3$. The case $a<0$ is by far more difficult. For any $d\ge2$, it has also been proved in~\cite{DELT09} that there is a curve $p\mapsto(a,b)$ taking values in the region $\{(a,b)\in\R^2\,:\,a<0\;\mbox{and}\;a<b<a+1\}$, which originates at $(a,b)=(0,0)$ when $p\to2^*$, such that $\lim_{p\to2}(a,b-a)=(-\infty,1)$, and which separates the region of symmetry from the region of symmetry breaking. When $d=2$, a non-explicit region of symmetry attached to $a=a_c=0$ has been obtained by a perturbation method in~\cite{MR2437030}. Perturbation results have also been obtained for $d\ge3$: see see~\cite{Lin-Wang-04,MR2053993} and~\cite[Theorem 4.8]{MR2001882}. Symmetry has been proved in \cite[Theorem~3.1]{MR1734159} when $a<0$ and $b>0$. In the case $a< 0$, the best known result so far in the region corresponding to $a<0$ and $b<0$ can be found in~\cite{DEL2011} where direct estimates show that symmetry holds under the condition
\[
b\ge\frac{d\,(d-1)+4\,d\,(a-a_c)^2}{6\,(d-1)+8\,(a-a_c)^2}+a-a_c=:b_{\,\rm direct}(a)\,.
\]

\smallskip To establish \emph{symmetry breaking} by perturbation is standard. One expands the functional
\[
\mathcal F[v]:=\mathsf C_{a,b}^\star\ird{\frac{|\nabla v|^2}{|x|^{2\,a}}}-\(\ird{\frac{|v|^p}{|x|^{b\,p}}}\)^{2/p}
\]
near the critical point $v_\star$ to second order by computing
\[
\mathsf Q[w]:=\lim_{\varepsilon\to0}\frac1{\varepsilon^2}\big(\mathcal F[v_\star+\varepsilon\,w]-\mathcal F[v_\star]\big)\,.
\]
The spectrum of the operator associated with the quadratic form $\mathsf Q$ determines the local stability or instability of the critical point $v_\star$. If $a<0$ and
\be{FS}
b<\frac{d\,(a_c-a)}{2\sqrt{(a_c-a)^2+d-1}}+a-a_c=:b_{\,\rm FS}(a)\,,
\ee
it turns that the lowest eigenvalue is negative and the radial optimal function is unstable, \emph{i.e.}, \emph{symmetry is broken}. The difference $b_{\,\rm direct}-b_{\,\rm FS}$ is of course nonnegative for any $a<0$ but corresponds to a remarkably small region: see Fig.~\ref{Figure}. If $b>b_{\,\rm FS}$ then the radial optimal function is locally stable. If $b=b_{\,\rm FS}$, the lowest spectral point of the operator associated with~$\mathsf Q$ is a zero eigenvalue, which incidentally determines $b_{\,\rm FS}$. The fact that symmetry can be broken was discovered by F.~Catrina and Z.-Q.~Wang in~\cite{Catrina-Wang-01}. The sharp condition given in~\eqref{FS} is due to V.~Felli and M.~Schneider in~\cite{Felli-Schneider-03} for $d\ge3$. Actually, the case $d=2$ is also covered by~\eqref{FS}. For brevity, we shall call it the \emph{Felli-Schneider region} and call the curve $b=b_{\,\rm FS}(a)$ the \emph{Felli-Schneider curve}. The issue of symmetry in the Caffarelli-Kohn-Nirenberg inequalities~\eqref{CKN} was studied numerically in~\cite{Freefem}. In~\cite{DE2012} formal expansions were used to establish the behavior of non-radial critical points near the bifurcation point supporting \emph{the conjecture that the Felli-Schneider curve is the threshold between the symmetry and the symmetry breaking region}. This is precisely what we prove in this paper.
\begin{thm}\label{Thm:Main1} Let $d\ge2$ and $p\in (2, 2^*)$. If either $a\in[0,a_c)$ and $b>0$, or $a<0$ and $b\ge b_{\,\rm FS}(a)$, then the optimal functions for the Caffarelli-Kohn-Nirenberg inequalities~\eqref{CKN} are radially symmetric. \end{thm}
\begin{figure}[h]
\includegraphics{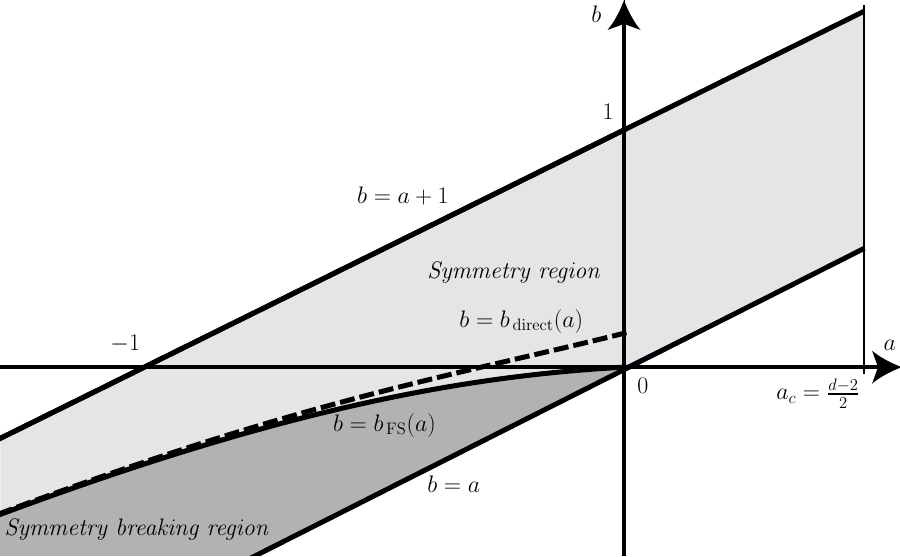}
\caption{\label{Figure}\small\emph{The \emph{Felli-Schneider region}, or symmetry breaking region, appears in dark grey and is defined by $a<0$, $a< b<b_{\,\rm FS}(a)$. We prove that symmetry holds in the light grey region defined by $b_{\,\rm FS}(a)\le b<a+1$ when $a<0$ and for any $b\in[a,a+1]$ if $a\in[0,a_c)$. The curve $a\mapsto b_{\,\rm direct}(a)$ corresponds to the dashed curve. The above plot is done with \hbox{$d=3$}.}}
\end{figure}

Note that for $(a,b)=(0,0)$ the Caffarelli-Kohn-Nirenberg inequalities are reduced to Sobolev's inequality. In this case the functional $\mathcal F$ has the larger, non-compact symmetry group $O(d+1,1)$. The celebrated Aubin-Talenti functions are optimal. The optimizers are unique up to multiplications by a constant, translations and scalings. For $(a,b)\not=(0,0)$ the functional $\mathcal F$ has the smaller symmetry group $O(d)$, \emph{i.e.}, it is invariant under rotations and reflections about the origin, and scalings. This symmetry persists for the optimizers in the parameter region $b \ge b_{\,\rm FS}(a)$, but is broken in the remaining \emph{Felli-Schneider region}. The optimizers have an $O(d-1)$ symmetry: see~\cite{MR2001882}. In this sense we have obtained the full phase diagram for the optimal functions of the Caffarelli-Kohn-Nirenberg inequalities. The reader may consult~\cite{Oslo} for a review of known results and~\cite{1406} for some recent progress.

Our method yields a stronger result than Theorem~\ref{Thm:Main1}, which can be interpreted as a \emph{rigidity result}. Consider the equation
\be{EL}
-\,\nabla\cdot\big(|x|^{-2\,a}\,\nabla v\big)=|x|^{-b\,p}\,|v|^{p-2}\,v\,.
\ee
\begin{thm}\label{Thm:Rigidity} Assume that $d\ge2$ and $p\in(2,2^*)$. If either $a\in[0,a_c)$ and $b>0$, or $a<0$ and $b\ge b_{\,\rm FS}(a)$, then any nonnegative solution $v$ of~\eqref{EL} which satisfies $\ird{\frac{|v|^p}{|x|^{b\,p}}}<\infty$ is equal to $v_\star$ up to a scaling.\end{thm}
This uniqueness result is not true anymore in the Felli-Schneider region of symmetry breaking: there we find at least two distinct nonnegative solutions, one radial and the other one non-radial. Our method of proof relies on a computation which is by many aspects similar to the one that can be found in~\cite{MR615628,BV-V} in the case of elliptic equations on compact manifolds, and without weights. Such results are called \emph{rigidity results} because they aim at proving that only trivial solutions may exist. Trivial solutions are replaced in our case by the radial solution $v_\star$.

We would also like to emphasize that our method does not rely on any kind of rearrangement technique. So far, it seems that symmetrization techniques simply do not work for the examples at hand.

\medskip The key idea of our method is to exhibit a \emph{nonlinear flow} under the action of which $\mathcal F$ is monotone non-increasing, and whose limit is $v_\star$. In practice, we do not need to take into account the whole flow, and it is enough to perturb a critical point of $\mathcal F$ in the infinitesimal direction indicated by the flow: we reach a contradiction if this critical point is not radially symmetric. Why the flow is the right tool to consider, at least at heuristic level, will be explained in Section~\ref{Sec:Flow}.

Before doing that, we will reduce the Caffarelli-Kohn-Nirenberg inequalities to \emph{Sobolev type inequalities} in which the dimension $n$ is not necessarily an integer. With respect to this ``dimension" the inequalities are critical. Alternatively, the inequalities can be seen as Sobolev type inequalities on $\R^d$, with a weight $|x|^{n-d}$.

Sharp constants in Sobolev type inequalities can be characterized as optimal decay rates of entropies under the action of a fast diffusion flow: see~\cite{MR1940370}. There is a by now standard method to prove this, which is a nonlinear version of the \emph{carr\'e du champ} method of D.~Bakry and M.~Emery, and whose strategy goes back to~\cite{MR772092,Bakry-Emery85}. The nonlinear version was studied in~\cite{MR1777035,MR1986060,MR2745814} in the case of the fast diffusion equation. The key idea is to prove that a nonlinear Fisher information and its derivative can be compared. The carr\'e du champ method takes curvature and weights very well into account. The identity, which encodes the Bochner-Lichnerowicz-Weitzenb\"ock formula in the context of semi-groups and Markov processes, is usually designated as the $\mathrm{CD}(\rho,N)$ condition and has been extensively used in the context of Riemannian geometry. See~\cite{MR3155209} for a detailed account. This book, and in particular~\cite[chapter~6]{MR3155209}, has also been a source of inspiration for the change of variables of Section~\ref{Sec:Flow} and for the idea to `change the dimension' from $d$ to~$n$.

The link between the \emph{carr\'e du champ} method and \emph{rigidity} results was established in~\cite{MR1412446} and later exploited for interpolation inequalities and evolution problems on manifolds in~\cite{MR2381156}. See~\cite{MR3229793,DEKL} for more recent and detailed results in this direction. However, in our setting, the relation of the Fisher information and its derivative along the flow is more a nonlinear effect than a curvature issue, as in~\cite{MR3200617,1501}. This will be made clear in Section~\ref{Sec:Flow}. As a last observation, let us mention that the optimal cases of interpolation on the sphere and on the line, which were obtained by the combination of a stereographic projection and the Emden-Fowler transformation in~\cite{DoEsLa}, are understood using nonlinear flow methods, but the interpolation on the \emph{cylinder}, which is closely connected with~\eqref{CKN} as we shall see next, was still open.

\medskip The Caffarelli-Kohn-Nirenberg inequalities~\eqref{CKN} on~$\R^d$ are equivalent to Gagliar\-do-Nirenberg interpolation inequalities on the cylinder $\mathcal C_1:=\R\times\S^{d-1}$. As was observed in~\cite{Catrina-Wang-01}, this follows from the \emph{Emden-Fowler transformation}
\be{EF}
v(r,\omega)=r^{a-a_c}\,\varphi(s,\omega)\quad\mbox{with}\quad r=|x|\,,\quad s=-\log r\quad\mbox{and}\quad\omega=\frac xr\,.
\ee
With this transformation, inequality~\eqref{CKN} can be rewritten as
\be{Ineq:Gen_interp_Cylinder}
\nrmcnd{\partial_s\varphi}2^2+\nrmcnd{\nabla_\omega\varphi}2^2+\Lambda\,\nrmcnd{\varphi}2^2\ge\mu(\Lambda)\,\nrmcnd{\varphi}p^2\quad\forall\,\varphi\in\mathrm H^1(\mathcal C_1)\,,
\ee
where $\Lambda:=(a_c-a)^2$ and, using~\eqref{exponent-relationabp}, the optimal constant $\mu(\Lambda)$ is
\[
\mu(\Lambda)=\frac1{\mathsf C_{a,b}}\quad\mbox{with}\quad a=a_c\pm\sqrt\Lambda\quad\mbox{and}\quad b=\frac d{p}\pm\sqrt\Lambda\,.
\]
Strictly speaking,~\eqref{exponent-relationabp} with $a<a_c$ is given by $a=a_c-\sqrt\Lambda$, but~\eqref{Ineq:Gen_interp_Cylinder} is independent of the sign of $a-a_c$, hence proving that~\eqref{exponent-relationabp} also holds with $a>a_c$: this is a proof of the modified inversion symmetry property. Notice that $\nabla_\omega$ denotes the gradient with respect to angular variables only, and we shall use the notation $\Delta_\omega$ for the Laplace-Beltrami operator on $\S^{d-1}$.

Radial symmetry of $v$ means that $\varphi$ depends only on~$s$ and we shall then say that $\varphi$ is \emph{symmetric}. Scaling invariance in $\R^d$ is equivalent to invariance under translations in the $s$-direction and any optimal function satisfies, with a proper normalization, the Euler-Lagrange equation
\be{eqlinder}
-\,\partial^2_s\,\varphi-\,\Delta_\omega\,\varphi+\Lambda\,\varphi=\varphi^{p-1}\quad\mbox{in}\quad\mathcal C_1\,.
\ee
The symmetric solution to~\eqref{eqlinder} is explicit and given by
\be{varphistar}
\varphi_\Lambda(s)=\beta\,\big(\cosh(\alpha\,s)\big)^{-\frac2{p-2}}\,,\quad\alpha=\tfrac{p-2}2\,\sqrt\Lambda\,,\quad\beta=\(\tfrac p2\,\Lambda\)^\frac1{p-2}\,.
\ee
Among symmetric solutions (solutions depending only on $s$), it is unique up to translations in the $s$-direction. The \emph{Felli-Schneider curve} is given in terms of the parameters $p=\frac{2\,n}{n-2}$ and $\Lambda$ by
\[
\Lambda=\Lambda_{\,\rm FS}:=4\,\frac{d-1}{p^2-4}\,.
\]
The results of Theorems~\ref{Thm:Main1} and~\ref{Thm:Rigidity} have an exact counterpart on the cylinder.
\begin{cor}\label{Cor:CylinderConstant} Assume that $d\ge2$ and $p\in(2,2^*)$. Equality in~\eqref{Ineq:Gen_interp_Cylinder} is achieved by~$\varphi_\Lambda$ if and only if $\Lambda\le\Lambda_{\,\rm FS}$. Moreover, up to translations, the unique solution to~\eqref{eqlinder} is equal to $\varphi_\Lambda$ if $\Lambda\le\Lambda_{\,\rm FS}$.\end{cor}
As a consequence, the value of the optimal constant is explicit for any $\Lambda\le\Lambda_{\,\rm FS}$ and given by $\mu(\Lambda)=|\S^{d-1}|^{1-\frac2p}\,\mu_\R(\Lambda)$ with
\[
\mu_\R(\Lambda)=\frac{\nrmr{\partial_s\varphi_\Lambda}2^2+\Lambda\,\nrmr{\varphi_\Lambda}2^2}{\nrmr{\varphi_\Lambda}p^2}=\tfrac p2\,\Lambda^\frac{p+2}{2\,p}\(\frac{2\,\sqrt\pi\;\Gamma\big(\frac p{p-2}\big)}{(p-2)\,\Gamma\big(\frac{3\,p-2}{2\,(p-2)}\big)}\)^\frac{p-2}p\,.
\]

\medskip Our method is not limited to the cylinder $\mathcal C_1=\R\times\S^{d-1}$ and provides rigidity results for a large class of non-compact manifolds. Let us consider the case of a general cylinder
\[
\mathcal C=\R\times\M
\]
where $(\M,g)$ is a smooth compact connected Riemannian manifold of dimension $d-1$, without boundary. Let us denote by $\Delta_g$ the Laplace-Beltrami operator on~$\M$. We denote by $d\kern1pt v_g$ the volume element. We shall also denote by~$\mathrm{Ric}_g$ the Ricci tensor and by~$\lambda_1^{\M}$ the lowest positive eigenvalue of $-\,\Delta_g$. We consider the case where the curvature of $\M$ is bounded from below and define
\[
\kappa:=\inf_{\M}\inf_{\xi\in\S^{d-2}}\mathrm{Ric}_g(\xi\,,\xi)\,,
\]
\[
\lambda_\theta:=\(1+\delta\,\theta\,\frac{d-1}{d-2}\)\kappa+\delta\,(1-\theta)\,\lambda_1^{\M}\quad\mbox{with}\quad\delta=\frac{n-d}{(d-1)\,(n-1)}\,,
\]
where the dependence of $\lambda_\theta$ on $\theta$ will be discussed in Remark~\ref{Rem:lambda}. With
\[
\lambda_\star:=\lambda_{\theta_\star}\quad\mbox{where}\quad\theta_\star:=\frac{(d-2)\,(n-1)\,\big(3\,n+1-d\,(3\,n+5)\big)}{(d+1)\,\big(d\,(n^2-n-4)-n^2+3\,n+2\big)}\quad\mbox{if}\quad d\ge2
\]
and $\theta_\star=0$ if $d=2$, let us state a \emph{rigidity} result for the equation
\be{Eqn:Cylinder}
-\,\partial^2_s\,\varphi-\,\Delta_g\,\varphi+\Lambda\,\varphi=\varphi^{p-1}\quad\mbox{in}\quad\mathcal C\,.
\ee
\begin{thm}\label{Thm:LV} Assume that $d\ge2$, $p\in(2,2^*)$ and $\Lambda>0$. If $\Lambda\le\frac{\lambda_\star}{p-2}$, then any positive solution $\varphi\in\mathrm H^1(\mathcal C)$ of~\eqref{Eqn:Cylinder} is equal to $\varphi_\Lambda$, up to a translation in the $s$-direction.
\end{thm}
The condition in Theorem~\ref{Thm:LV} is reminiscent of the result of J.R.~Licois and L.~V\'eron in~\cite{MR1338283,MR1631581} which uses an interpolation between $\lambda_1^{\M}$ and $\kappa$. In our case, we can take $\theta=\theta_\star<0$.

To conclude the introduction, we come back to the question of the Gagliardo-Nirenberg interpolation inequalities on a general cylinder $\mathcal C$
\be{ineqcylindergeneral}
\nrmcndgen{\partial_s\varphi}2^2+\nrmcndgen{\nabla_g\varphi}2^2+\Lambda\,\nrmcndgen{\varphi}2^2\ge\mu(\Lambda)\,\nrmcndgen{\varphi}p^2\quad\forall\,\varphi\in\mathrm H^1(\mathcal C)
\ee
where $\mu(\Lambda)$ denotes the optimal constant for any $\Lambda>0$. As an extension of the constant found by V.~Felli and M.~Schneider, we define
\[
\Lambda_{\,\rm FS}:=\frac{4\,\lambda_1^{\M}}{p^2-4}
\]
for a general manifold $\M$. Notice that $\lambda_1^{\M}=d-1=\kappa+1$ if $\M=\S^{d-1}$ and, as a consequence in this case, $\frac{\lambda_\star}{p-2}=\Lambda_{\,\rm FS}$. Let us define
\[
\mu_\star(\Lambda):=\big(\mathrm{vol}_g(\M)\big)^{1-\frac2p}\,\mu_\R(\Lambda)\,.
\]
\begin{cor}\label{Cor:CylinderInequality} Assume that $d\ge2$ and $p\in(2,2^*)$. The optimal constant in~\eqref{ineqcylindergeneral} is given by $\mu(\Lambda)=\mu_\star(\Lambda)$ if and only if $\Lambda\in(0,\Lambda_\star]$ where $\Lambda_\star$ is such that
\[
\frac{\lambda_\star}{p-2}\le\Lambda_\star\le\Lambda_{\,\rm FS}\,,
\]
and if $\Lambda_\star<\Lambda_{\,\rm FS}$, then $\frac{\lambda_\star}{p-2}<\Lambda_\star<\Lambda_{\,\rm FS}$. \end{cor}

Our paper is organized as follows. We first generalize the linear instability result of V.~Felli and M.~Schneider in Section~\ref{Sec:FS}. In Section~\ref{Sec:Sobolev}, we reformulate our problem as a critical Sobolev inequality in a space with a dimension $n>d$, which allows us to use tools based on the fast diffusion equation in Section~\ref{Sec:Flow}, at least at heuristic level. The results of Section~\ref{Sec:Intro} are proved in Section~\ref{Sec:Proofs}. They rely on a key technical result which is stated in Corollary~\ref{Cor:DerivFisherSign}. Our methods also yield new sharp spectral estimates on cylinders (which were announced in \cite{dolbeault:hal-01137403}). Moreover, a precise version of an improved Hardy's inequality is obtained. These results are proved in Section~\ref{Sec:Consequences}.

\section{Linear instability of symmetric critical points}\label{Sec:FS}

On~$\mathcal C_1=\R\times\S^{d-1}$, we consider the measure $d\nu=ds\,d\omega$, where $d\omega$ is the measure induced by Lebesgue's measure on $\S^{d-1}\subset\R^d$. On a general cylinder~$\mathcal C=\R\times\M$, we consider the measure $d\nu=ds\,d\kern1pt v_g$ and still denote by $\omega$ the generic variable on $\M$. Since Caffarelli-Kohn-Nirenberg inequalities~\eqref{CKN} are equivalent to Gagliardo-Nirenberg inequalities~\eqref{Ineq:Gen_interp_Cylinder} on~$\mathcal C_1$ with $\Lambda:=(a_c-a)^2$ and since solutions of~\eqref{EL} are transformed into solutions to~\eqref{eqlinder} by the Emden-Fowler transformation~\eqref{EF}, we will work directly in the general cylinder setting, that is, on $\mathcal C$. Minor modifications of an existence argument proved in \cite{Catrina-Wang-01} for extremals of  \eqref{Ineq:Gen_interp_Cylinder} yields the existence of optimal functions for~\eqref{ineqcylindergeneral} for any $p\in(2,2^*)$ with $2^*=\frac{2\,d}{d-2}$ if $d\ge3$ and $2^*=\infty$ if $d=1$ or $2$. Up to multiplication by a constant, all extremal functions are positive on $\mathcal C$.

Is $s\mapsto\varphi_\Lambda(s)$ as defined by~\eqref{varphistar} and seen as a function on $\mathcal C$ optimal for~\eqref{Ineq:Gen_interp_Cylinder} ? Equivalently, is $\varphi_\Lambda$ a minimizer of
\[
\varphi\mapsto\mathcal G[\varphi]:=\nrmcndgen{\partial_s\varphi}2^2+\nrmcndgen{\nabla_\omega\varphi}2^2+\Lambda\,\nrmcndgen{\varphi}2^2-\mu_\star(\Lambda)\,\nrmcndgen{\varphi}p^2\,?
\]
This can be tested by perturbing $\mathcal G$ around $\varphi_\Lambda$. Here we simply extend the strategy of \cite{Catrina-Wang-01,Felli-Schneider-03} to a general cylinder $\mathcal C$. Since the ground state of the Schr\"odinger operator $-\,\partial^2_s\,+\,\Lambda-\,(p-1)\,\mu_\star(\Lambda)\,\varphi_\Lambda^{p-2}$ is generated by $\varphi_\Lambda^{p/2}$, we may consider $\mathcal G\big[\varphi_\Lambda+\varepsilon\,\varphi_\Lambda^{p/2}\,\phi_1\big]$ as $\varepsilon\to0$, where $\phi_1$ is an eigenfunction associated with the first positive eigenvalue~$\lambda_1^{\M}$ of $-\,\Delta$ on $\M$. An elementary computation shows that
\[
\mathcal G\left[\varphi_\Lambda+\varepsilon\,\varphi_\Lambda^{p/2}\,\phi_1\right]=\varepsilon^2\,\nrmr{\varphi_\Lambda}p^p\(\lambda_1^{\M}-\,\tfrac14\,(p^2-4)\,\Lambda\)+o(\varepsilon^2)\,.
\]
\begin{prop}\label{Prop:FS} With the above notations, $\varphi_\Lambda$ is \emph{not} a local minimizer of $\mathcal G$ if $\Lambda>4\,\lambda_1^{\M}/(p^2-4)$.\end{prop}
One can actually check that Proposition~\ref{Prop:FS} states the sharp condition for linear instability of $\varphi_\Lambda$. Details are left to the reader. See \cite{dolbeault:hal-01137403} for a similar computation for spectral estimates, and \cite{DE2012} for an expansion of the non-symmetric branch around the bifurcation point. Hence the region of linear instability of symmetric critical points is given by $\alpha>\alpha_{\,\rm FS}$, with
\[
\alpha_{\,\rm FS}:=\sqrt{\frac{p-2}{p+2}\,\lambda_1^{\M}}\,.
\]
In terms of $\Lambda$, the above condition is equivalent to \hbox{$\Lambda>\Lambda_{\,\rm FS}:=\frac{4\,\lambda_1^{\M}}{p^2-4}$}.
When $\M=\S^{d-1}$, $\lambda_1^{\M}=d-1$ and we recover the expression of $\Lambda_{\,\rm FS}$ found in \cite{Felli-Schneider-03}.

\section{A change of variables and a Sobolev type inequality}\label{Sec:Sobolev}

The first step of our method is a change of variables which reduces the Caffarelli-Kohn-Nirenberg inequalities to a Sobolev type inequality in a non-integer dimension $n>d$. A similar transformation can also be used for Gagliardo-Nirenberg inequalities on cylinders, combined with an inverse Emden-Fowler transformation.

\subsection{The case of Caffarelli-Kohn-Nirenberg inequalities}\label{Sec:CKN-Sobolev}
We start by proving that~\eqref{CKN} is equivalent to a Sobolev type inequality with a weight. From now on, we also assume that $d\ge2$. Written in spherical coordinates, with
\[
r=|x|\quad\mbox{and}\quad\omega=\frac x{|x|}\,,
\]
the Caffarelli-Kohn-Nirenberg inequality~\eqref{CKN} becomes
\[
\(\irdsph{|v|^p}{d-b\,p}\)^\frac2p\le\mathsf C_{a,b}\irdsph{\left|\nabla v\right|^2}{d-2\,a}
\]
where $\left|\nabla v\right|^2=\left|\tfrac{\partial v}{\partial r}\right|^2+\tfrac1{r^2}\,\left|\nabla_\omega v\right|^2$ and $\nabla_\omega$ denotes the gradient with respect to the angular variable $\omega\in\S^{d-1}$. Next we consider the change of variables $r\mapsto r^\alpha$,
\be{wv}
v(r,\omega)=w(r^\alpha,\omega)\quad\forall\,(r,\omega)\in\R^+\times\S^{d-1}
\ee
so that
\begin{multline*}
\alpha^{1-\frac2p}\(\irdsph{|w|^p}{\frac{d-b\,p}\alpha}\)^\frac2p\\
\le\mathsf C_{a,b}\irdsph{\(\alpha^2\left|\tfrac{\partial w}{\partial r}\right|^2+\tfrac1{r^2}\,|\nabla_\omega w|^2\)}{\frac{d-2\,a-2}\alpha+2}\,,
\end{multline*}
and pick $\alpha$ so that
\[
n=\frac{d-b\,p}\alpha=\frac{d-2\,a-2}\alpha+2\,.
\]
Hence, we define new parameters
\[
\alpha=\frac{(1+a-b)\,(a_c-a)}{a_c-a+b}\quad\mbox{and}\quad n=\frac{2\,p}{p-2}=\frac d{1+a-b}\,.
\]
If we think of $n$ as a non-integer dimension, then $p=\frac{2\,n}{n-2}$ is the associated critical Sobolev exponent. Since $-\infty<a<a_c$, $0<b-a<1$ and $2<p<\frac{2\,d}{d-2}$, the parameters $\alpha$ and $n$ vary in the ranges $0<\alpha<\infty$ and $d<n<\infty$. The \emph{Felli-Schneider curve} in the $(\alpha, n)$ variables is given by
\[
\alpha=\sqrt{\frac{d-1}{n-1}}=:\alpha_{\,\rm FS}\,.
\]
Hence, the region of symmetry breaking is given by $\alpha>\alpha_{\,\rm FS}$. In the new variables, the derivatives are given by
\[
\D w=\(\alpha\,\frac{\partial w}{\partial r},\frac1r\,\nabla_\omega w\)\,.
\]
On $\R^d\approx(0,+\infty)\times\S^{d-1}$, we consider the measure
\[
d\mu:=r^{n-1}\,dr\,d\omega\,.
\]
The inequality becomes
\be{CKN2}
\alpha^{1-\frac2p}\(\irdmu{|w|^p}\)^\frac2p\le\mathsf C_{a,b}\irdmu{|\D w|^2}
\ee
and has the homogeneity of Sobolev's inequality for functions defined on $\R^n$ if~$n$ is an integer. The result of Theorem~\ref{Thm:Main1} can be rephrased as follows.
\begin{prop}\label{Thm:Sobolev} Let $d\ge 2$ and assume that $\alpha\le \alpha_{\rm FS}$. Then optimality in~\eqref{CKN2} is achieved by radial functions.\end{prop}
The r.h.s.~in~\eqref{CKN2} generically differs from the usual Dirichlet integral because of the coefficient $\alpha$ in the derivative $\D$ and because the angular variable $\omega$ is in $\S^{d-1}$ with $d<n$.

\medskip\noindent\emph{Notations.} When there is no ambiguity, we will omit the index $_\omega$ and from now on write that $\nabla=\nabla_\omega$ denotes the gradient with respect to the angular variable $\omega\in\S^{d-1}$ and that $\Lap$ is the Laplace-Beltrami operator on $\S^{d-1}$. We define the self-adjoint operator $\L$ by
\[
\L w:=-\,\D^*\,\D\,w=\alpha^2\,w''+\alpha^2\,\frac{n-1}r\,w'+\frac{\Delta\,w}{r^2}\,.
\]
The fundamental property of $\L$ is the fact that
\[
\irdmu{w_1\,\L w_2}=-\irdmu{\D w_1\cdot\D w_2}\quad\forall\,w_1,\,w_2\in\mathcal D(\R^d)\,,
\]

\subsection{The case of Gagliardo-Nirenberg inequalities on general cylinders}\label{Sec:GN-Sobolev}

If we study solutions to~\eqref{Eqn:Cylinder} or inequality~\eqref{ineqcylindergeneral}, the strategy is to use the inverse Emden-Fowler transform to rewrite the problem on $(0,\infty)\times\M$ and then use the change of variables $r\mapsto r^\alpha$ as in Section~\ref{Sec:CKN-Sobolev} to write a Sobolev type inequality. Let us consider the change of variables
\be{wvarphi}
\varphi(s,\omega)=e^\frac{2\,s}{p-2}\,w\(e^{-\,\alpha\,s},\omega\)\quad\forall\,(s,\omega)\in\mathcal C=\R\times\M\,,
\ee
with
\[
\alpha=\frac{p-2}2\,\sqrt\Lambda\quad\mbox{and}\quad n=\frac{2\,p}{p-2}\,.
\]
Inequality~\eqref{ineqcylindergeneral} is then equivalent to
\be{CKN3}
\mu(\Lambda)\,\alpha^{1-\frac2p}\(\irdmugen{|w|^p}\)^\frac2p\le\irdmugen{|\D w|^2}
\ee
where $d\mu:=r^{n-1}\,dr\,d\kern1pt v_g$ is a measure on $(0,\infty)\times\M$ and $\D w=\(\alpha\,\frac{\partial w}{\partial r},\frac1r\,\nabla_g w\)$ where $\nabla_g$ denotes the gradient on $\M$. Inequality~\eqref{CKN3} coincides with~\eqref{CKN2} when $\M=\S^{d-1}$. If $\Delta_g$ denotes the Laplace-Beltrami operator on $\M$, then \eqref{Eqn:Cylinder} can be rewritten as
\[
-\,\D^*\,\D\,w+w^{p-1}=0\quad\mbox{in}\quad(0,\infty)\times\M\,.
\]
The region where we shall prove symmetry is given by
\begin{equation}\label{Range}
\alpha\le\frac12\,\sqrt{(p-2)\,\lambda_\star}
\end{equation}
with $\lambda_\star$ defined in the introduction. The symmetry region coincides with $\alpha\le\alpha_{\,\rm FS}$ when $\M=\S^{d-1}$. Conversely, symmetry is broken if $\alpha>\alpha_{\,\rm FS}$ with
\[
\alpha_{\,\rm FS}=\frac{p-2}2\,\sqrt{\Lambda_{\,\rm FS}}=\sqrt{\frac{p-2}{p+2}\,\lambda_1^{\M}}\,.
\]

\medskip\noindent\emph{Notations.} When there is no ambiguity, we shall write $\nabla_g=\nabla$ and $\Delta_g=\Delta$ in what follows.

\section{Heuristics: monotonicity along a well chosen nonlinear flow}\label{Sec:Flow}

In this section we collect some observations on the monotonicity of a generalized Fisher information along a fast diffusion flow. These observations explain our strategy. We consider the measure $d\mu=r^{n-1}\,dr\,d\kern1pt v_g$ on $(0,\infty)\times\M$.

Let us start with the Fisher information. We transform the Sobolev type inequality~\eqref{CKN3} of the previous section as follows. With
\be{uw}
u^{\frac 12-\frac1n}=|w|\quad\Longleftrightarrow\quad u=|w|^p\quad\mbox{with}\quad p=\frac{2\,n}{n-2}
\ee
the r.h.s.~in~\eqref{CKN2} is transformed into a generalized \emph{Fisher information}
\be{Fisher}
\mathcal I[u]:=\irdmugen{u\,|\D\mathsf p|^2}\quad\mbox{where}\quad\mathsf p=\frac m{1-m}\,u^{m-1}\quad\mbox{and}\quad m=1-\frac1n\,,
\ee
while the l.h.s.~in~\eqref{CKN2} is now proportional to a \emph{mass}, $\irdmugen u$. Here $\mathsf p$ is the \emph{pressure function}, as in~\cite[5.7.1 p.~98]{vazquez2004asymptotic}. If we replace $m$ by $1-\frac1n$, we get that
\be{up}
\mathsf p=(n-1)\,u^{-\frac1n}
\ee
is such that $\D u^m=-\,u\,\D\mathsf p$ and $\frac{\D u}u=-\,n\,\frac{\D\mathsf p}{\mathsf p}$. The reader is invited to check that
\begin{multline*}
\irdmugen{\(\alpha^2\,\left|\frac{\partial w}{\partial r}\right|^2+\frac1{r^2}\,\left|\nabla w\right|^2\)}=\tfrac14\,\big(\tfrac{n-2}{n-1}\big)^2\,\mathcal I[u]\\
\mbox{and}\quad\irdmugen{|w|^p}=\irdmugen u\,.
\end{multline*}
For later purpose, let $\eta:=4\,\big(\tfrac{n-1}{n-2}\big)^2\,\alpha^{1-\frac2p}$. Collecting these considerations, we have shown that the optimal constant in~\eqref{CKN3} can be characterized as follows.
\begin{prop}\label{Lem:MinimizationFisher} With the above notations and $\Lambda=4\,\alpha^2/(p-2)^2$, we have
\be{Eqn:InfFisher}
\mu(\Lambda)=\eta^{-1}\,\inf\left\{\mathcal I[u]\,:\,\irdmugen u=1\right\}\,.
\ee
\end{prop}
We are interested not only in minimizers, but also in critical points of $\mathcal I$ under the mass constraint.

Next we consider the \emph{fast diffusion equation} in $\R^d$, with \hbox{$d\ge2$}, given~by
\be{FDE}
\frac{\partial u}{\partial t}=\L u^m\,,\quad m=1-\frac1n\,.
\ee
At a heuristic level,~\eqref{FDE} preserves the mass and decreases the Fisher information when the parameters are in the symmetry range determined by~\eqref{Range}. In this range, the decay of the Fisher information is strict, except for self-similar solutions which correspond to symmetric critical points of $\mathcal I$ under a mass constraint, as we shall see below. We will actually use the flow only to characterize the direction in which one has to perturb an arbitrary critical point. To understand why such a direction has to be considered, it is useful to use the quantities provided by the nonlinear flow. Let us give some details.

\medskip Equation~\eqref{FDE} admits self-similar solutions of Barenblatt type, which are given~by
\[
u_\star(t,r,\omega)=t^{-n}\(\mathsf c_\star+\frac{r^2}{2\,(n-1)\,\alpha^2\,t^2}\)^{-n}\quad\forall\,(t,r,\omega)\in(0,\infty)\times(0,\infty)\times\M\,.
\]
The constant $c_\star$ is a numerical constant which has to be adjusted so that \hbox{$\irdmugen{u_\star}=1$}. It has an explicit value and depends only on $d$ and $n$. Also notice that the variable $t$ plays the role of a scaling parameter. Except in this section, in which we deal with an evolution problem for heuristical reasons, $t$ has to be understood in the sense of a positive scale.

If we assume that the solution to~\eqref{FDE} is supplemented with an initial datum $u(t=0,x)=u_0(x)\ge0$ such that $\irdmugen{u_0}=1$, then it makes sense to consider a solution to the Cauchy problem which preserves the mass, that is, such that
\[
\frac d{dt}\irdmugen u=0\,,
\]
as was done for the classical fast diffusion equation in~\cite{MR797051}, and consider $\frac d{dt}\mathcal I[u(t,\cdot)]$.

The functional $\mathcal I$ is invariant under scalings. Indeed, let $\lambda$ be an arbitrary positive real number. If we consider $u_\lambda(x)=\lambda^d\,u(\lambda\,x)$ for any $x\in\R^d$, we get that $\mathcal I[u_\lambda]=\mathcal I[u]$ for any $\lambda>0$. As a special case, when $u=u_\star$, it is clear that $\mathcal I[u_\star]$ is independent of $t>0$. In the symmetry range, the function $u_\star$ is optimal for~\eqref{Eqn:InfFisher}. In any case, we have the following characterization.
\begin{prop}\label{Lem:MinimizationFisherSymmetric} With the above notations and $\Lambda=4\,\alpha^2/(p-2)^2$, we have
\[
\eta\,\mu_\star(\Lambda)=\mathcal I[u_\star(t,\cdot)]\quad\forall\,t>0\,.
\]
\end{prop}
This result is easy to prove (details are left to the reader) and the question is to know under which conditions we also have $\mathcal I[u_\star(t,\cdot)]=\eta\,\mu(\Lambda)$, \emph{i.e.}, $\mu(\Lambda)=\mu_\star(\Lambda)$.

Our strategy for proving Theorem~\ref{Thm:Main1} is to establish that in the range~\eqref{Range}, the converse is also true, namely that $\frac d{dt}\mathcal I[u(t,\cdot)]=0$ implies that $u=u_\star$ up to a time shift, that is, up to a rescaling. Heuristically, this can be done as follows. If $u$ solves~\eqref{FDE}, then the \emph{pressure function} $\mathsf p$ given by~\eqref{up} solves
\be{Eqn:v}
\frac{\partial\mathsf p}{\partial t}=\frac1n\,\mathsf p\,\L\mathsf p-|\D\mathsf p|^2\,.
\ee
Let us define
\be{QK}
\Q[\mathsf p]:=\frac12\,\L\,|\D\mathsf p|^2-\D\mathsf p\cdot\D\L\mathsf p\quad\mbox{and}\quad\mathcal K[\mathsf p]:=\irdmugen{\(\Q[\mathsf p]-\frac1n\,(\L\mathsf p)^2\)\mathsf p^{1-n}}\,.
\ee
In order to handle boundary terms, we also define
\be{b}
\mathsf b(r):=r^{n-1}\iM{\Big(\tfrac\partial{\partial r}\(\mathsf p^{1-n}\,|\D\mathsf p|^2\)-\,\frac2n\,\mathsf p^{1-n}\,\mathsf p'\,\L\mathsf p\Big)}\,.
\ee
\begin{lem}\label{Lem:DerivFisher} With the notations defined by~\eqref{up},~\eqref{QK} and~\eqref{b}, if $u$ is a smooth solution of~\eqref{FDE}, with $\alpha\le \alpha_{\rm FS}$, and if $\lim_{r\to0_+}\mathsf b(r)=0=\lim_{r\to+\infty}\mathsf b(r)$ for any~$t>0$, then
\[\label{BLWevol}
\frac d{dt}\mathcal I[u(t,\cdot)]=-\,2\,(n-1)^{n-1}\,\mathcal K[\mathsf p(t,\cdot)]\,.
\]
\end{lem}

\begin{proof} All integrals are taken on $(0,\infty)\times\M$. Using~\eqref{FDE} and~\eqref{Eqn:v}, we can compute
\begin{eqnarray*}
\frac d{dt}\irdmugenn{u\,|\D\mathsf p|^2}&=&\irdmugenn{\frac{\partial u}{\partial t}\,|\D\mathsf p|^2}+\,2\irdmugenn{u\,\D\mathsf p\cdot\D\frac{\partial\mathsf p}{\partial t}}\\
&=&\irdmugenn{\L(u^m)\,|\D\mathsf p|^2}+\,2\irdmugenn{u\,\D\mathsf p\cdot\D\Big(\frac1n\,\mathsf p\,\L\mathsf p-|\D\mathsf p|^2\Big)}\,.
\end{eqnarray*}
If we omit all boundary terms, we get that
\begin{eqnarray*}
\frac d{dt}\irdmugenn{u\,|\D\mathsf p|^2}&=&\irdmugenn{u^m\,\L|\D\mathsf p|^2}+\,\frac2n\irdmugenn{u\,\mathsf p\,\D\mathsf p\cdot\D\L\mathsf p}\\
&&\qquad+\,\frac2n\irdmugenn{u\,\D\mathsf p\cdot\D\mathsf p\,\L\mathsf p}-\,2\irdmugenn{u\,\D\mathsf p\cdot\D\,|\D\mathsf p|^2}\\
&=&-\irdmugenn{u^m\,\L|\D\mathsf p|^2}+\,\frac2n\irdmugenn{u\,\mathsf p\,\D\mathsf p\cdot\D\L\mathsf p}\\
&&\qquad+\,\frac2n\irdmugenn{u\,\D\mathsf p\cdot\D\mathsf p\,\L\mathsf p}
\end{eqnarray*}
where the last line is given by an integration by parts:
\[
\irdmugenn{u\,\D\mathsf p\cdot\D\,|\D\mathsf p|^2}=-\irdmugenn{\D(u^m)\cdot\D\,|\D\mathsf p|^2}=\irdmugenn{u^m\,\L|\D\mathsf p|^2}\,.
\]
1) By definition of $\mathcal Q$, we get that
\[
\irdmugenn{u^m\,\L|\D\mathsf p|^2}=2\irdmugenn{u^m\,\Q[\mathsf p]}+2\irdmugenn{u^m\,\D\mathsf p\cdot\D\L\mathsf p}\,.
\]
2) Since $u\,\D\mathsf p=-\,\D(u^m)$, an integration by parts gives
\begin{multline*}
\irdmugenn{u\,\D\mathsf p\cdot\D\mathsf p\,\L\mathsf p}=-\irdmugenn{\D(u^m)\cdot\D\mathsf p\,\L\mathsf p}\\
=\irdmugenn{u^m\,(\L\mathsf p)^2}+\irdmugenn{u^m\,\D\mathsf p\cdot\D\L\mathsf p}
\end{multline*}
and with $u\,\mathsf p=(n-1)\,u^m$ we find that
\begin{multline*}
\frac2n\irdmugenn{u\,\mathsf p\,\D\mathsf p\cdot\D\L\mathsf p}+\,\frac2n\irdmugenn{u\,\D\mathsf p\cdot\D\mathsf p\,\L\mathsf p}\\
=\frac2n\irdmugenn{u^m\,(\L\mathsf p)^2}+2\irdmugenn{u^m\,\D\mathsf p\cdot\D\L\mathsf p}\,.
\end{multline*}
Collecting terms establishes Lemma~\ref{Lem:DerivFisher}.

Of course, in the previous computations, when integrating by parts, boundary terms have to be taken into account. Integrals are first taken on the domain $(r,R)\times\M$ with $0<r<R$. When doing integrations by parts, we get that all boundary terms are $\mathsf b(R)$ and $\mathsf b(r)$. The result follows from our assumptions.
\end{proof}

A natural approach is to show that $\mathcal I[u(t,\cdot)]$ is nonincreasing by proving that $\mathcal K[\mathsf p(t,\cdot)]$ is nonnegative for any $t>0$ in the range~\eqref{Range}. By Lemma~\ref{Lem:DerivFisher}, it is enough to prove that $\lim_{r\to0_+}\mathsf b(r)=0=\lim_{r\to+\infty}\mathsf b(r)$ for any $t>0$. As a consequence the flow drives any initial condition towards a minimum of $\mathcal I$, which is a self-similar solution of Barenblatt type, \emph{i.e.}, $u_\star$ up to a translation with respect to $t$. Hence $\mathcal I[u]\ge\mathcal I[u_\star(t,\cdot)]$, which is equivalent to the sharp form of inequality \eqref{CKN}.

We can avoid considering the flow by focusing on critical points for $\mathcal I[u]$ under the mass constraint. The Euler-Lagrange equation integrated against $\mathcal L\,u^m$ must vanish, provided that the integral exists. Heuristically, this quantity coincides with the time derivative of $\mathcal I[u(t,\cdot)]$ at the critical point. By Lemma~\ref{Lem:DerivFisher}, $\mathcal K[\mathsf p]$ must then vanish for critical points. Since $u$ solves an elliptic equation, we have additional regularity and decay properties of which we can take advantage to get rid of the boundary terms. With slight technical modifications, this is the line of arguments we shall use to prove the main results of the paper in Section \ref{Sec:Proofs}.

\section{The key computations}\label{Sec:Computations}

The goal of this section is to characterize the functions such that $\mathcal K[\mathsf p]=0$: see Corollary~\ref{Cor:DerivFisherSign}. This result is the main ingredient of our method.

\subsection{A preliminary computation}

The calculations below are carried out for a function $\mathsf p$ defined on $(0,\infty)\times\M$ where $\M$ is a $d-1$ dimensional smooth, compact Riemannian manifold. Here $'$ and $\nabla$ respectively denote the derivative with respect to $r$ and the gradient on $\M$, $g$ is the metric tensor, $\Delta$ the Laplace-Beltrami operator and $dv_g$ the volume element. We recall that
\[
\mathcal L\,\mathsf p=\alpha^2\,\mathsf p''+\alpha^2\,\frac{n-1}r\,\mathsf p'+\frac{\Delta\,\mathsf p}{r^2}\,,
\]
\[
\D\mathsf p\cdot\D w=\alpha^2\,\mathsf p'w'+\frac{\nabla\mathsf p\cdot\nabla w}{r^2}\quad\mbox{and}\quad|\D\mathsf p|^2=\alpha^2\,|\mathsf p'|^2+\frac{|\nabla\mathsf p|^2}{r^2}\,.
\]
We also define
\[
\kk[\mathsf p]:=\mathcal Q(\mathsf p)-\frac1n\,(\mathcal L\,\mathsf p)^2=\frac12\,\mathcal L\,|\D\mathsf p|^2-\D\mathsf p\cdot\D\,\mathcal L\,\mathsf p-\frac1n\,(\mathcal L\,\mathsf p)^2\,,
\]
and
\[
\kk_{\M}[\mathsf p]:=\frac12\,\Delta\,|\nabla\mathsf p|^2-\nabla\mathsf p\cdot\nabla\Delta\,\mathsf p-\tfrac1{n-1}\,(\Delta\,\mathsf p)^2-(n-2)\,\alpha^2\,|\nabla\mathsf p|^2\,.
\]
\begin{lem}\label{Lem:Derivmatrixform1} Let $n\not=1$ be any real number, $d\in\N$, $d\ge2$, and consider a function $\mathsf p\in C^3((0,\infty)\times\M)$, where $(\M,g)$ is a smooth, compact Riemannian manifold. Then we have
\[
\kk[\mathsf p]=\alpha^4\(1-\frac1n\)\left[\mathsf p''-\frac{\mathsf p'}r-\frac{\Delta\,\mathsf p}{\alpha^2\,(n-1)\,r^2}\right]^2+2\,\alpha^2\,\frac1{r^2}\left|\nabla\mathsf p'-\frac{\nabla\mathsf p}r\right|^2+\frac1{r^4}\,\kk_{\M}[\mathsf p]\,.
\]\end{lem}
\begin{proof} By definition of $\kk[\mathsf p]$, we have
\begin{eqnarray*}
\kk[\mathsf p]&=&\textstyle\frac{\alpha^2}2\left[\alpha^2\,\mathsf p'^2+\frac{|\nabla\mathsf p|^2}{r^2}\right]''+\frac{\alpha^2}2\frac{n-1}r
\left[\alpha^2\,\mathsf p'^2+\frac{|\nabla\mathsf p|^2}{r^2}\right]'+\frac1{2\,r^2}\,\Delta\left[\alpha^2\,\mathsf p'^2+\frac{|\nabla\mathsf p|^2}{r^2}\right]\\
&&\textstyle-\,\alpha^2\,\mathsf p'\(\alpha^2\,\mathsf p''+\alpha^2\,\frac{n-1}r\,\mathsf p'+\frac{\Delta\,\mathsf p}{r^2}\)'-\frac1{r^2}
\nabla\mathsf p\cdot\nabla\(\alpha^2\,\mathsf p''+\alpha^2\,\frac{n-1}r\,\mathsf p'+\frac{\Delta\,\mathsf p}{r^2}\)\\
&&\textstyle-\frac1n\(\alpha^2\,\mathsf p''+\alpha^2\,\frac{n-1}r\,\mathsf p'+\frac{\Delta\,\mathsf p}{r^2}\)^2\,,
\end{eqnarray*}
which can be expanded as
\begin{eqnarray*}
&&\hspace*{-18pt}\textstyle\frac{\alpha^2}2\left[ 2\,\alpha^2\,\mathsf p''^2+2\,\alpha^2\,\mathsf p'\,\mathsf p'''+2\,\frac{|\nabla\mathsf p'|^2+\nabla\mathsf p\cdot\nabla\mathsf p''}{r^2}
-8\,\frac{\nabla\mathsf p\cdot\nabla\mathsf p'}{r^3}+6\,\frac{|\nabla\mathsf p|^2}{r^4}\right]\\
&&\textstyle\hspace*{-18pt}+\,\alpha^2\,\frac{n-1}r\left[\alpha^2\,\mathsf p'\,\mathsf p''+\frac{\nabla\mathsf p\cdot\nabla\mathsf p'}{r^2}-\frac{|\nabla\mathsf p|^2}{r^3}\right]+\frac1{r^2}\left[\alpha^2\,\mathsf p'\Delta\,\mathsf p'+\alpha^2\,|\nabla\mathsf p'|^2+\frac{\Delta\,|\nabla\mathsf p|^2}{2\,r^2}\right]\\
&&\textstyle\hspace*{-18pt}-\,\alpha^2\,\mathsf p'\(\alpha^2\,\mathsf p'''+\alpha^2\,\frac{n-1}r\,\mathsf p''-\,\alpha^2\,\frac{n-1}{r^2}\mathsf p'-2\,\frac{\Delta\,\mathsf p}{r^3}+\frac{\Delta\,\mathsf p'}{r^2}\)\\
&&\hspace*{2cm}\textstyle-\frac1{r^2}
\(\alpha^2\,\nabla\mathsf p\cdot\nabla\mathsf p''+\alpha^2\,\frac{n-1}r\nabla\mathsf p\cdot\nabla\mathsf p'+\frac{\nabla\mathsf p\cdot\nabla\Delta\,\mathsf p}{r^2}\)\\
&&\textstyle\hspace*{-18pt}-\frac1n\left[\alpha^4\,\mathsf p''^2+\alpha^4\,\frac{(n-1)^2}{r^2}\,\mathsf p'^2+\frac{(\Delta\,\mathsf p)^2}{r^4}+2\,\alpha^4\,\frac{n-1}r\,\mathsf p'\,\mathsf p''+2\,\alpha^2\,\frac{\mathsf p''\Delta\,\mathsf p}{r^2}+2\,\alpha^2\,\frac{n-1}{r^3}\mathsf p'\Delta\,\mathsf p\right].
\end{eqnarray*}
By ordering the terms in powers of $\alpha$, we get the result.
\end{proof}

\subsection{An identity if \texorpdfstring{$d\ge3$}{d ge 3}}\label{Subsec:dge3}
On the smooth compact Remannian manifold $(\M,g)$ we denote by $\mathrm H\fv$ the \emph{Hessian} of $\fv$, i.e, $\mathrm H\fv_{i;j}=\nabla_i\,\partial_j\fv$ where $\nabla_j$ denotes the covariant derivative. Thus $\mathrm H\fv$ is a symmetric covariant tensor of rank $2$. With a slight abuse of language we identify its trace with the Laplace-Beltrami operator
\[
\Delta\,f=\mathrm{Tr}\,(\mathrm H\fv)=\sum_{i,j} g^{i,j}\,\mathrm H\fv_{j;i}
\]
where, as usual, $\sum_j g^{i,j}g_{j,k}=\delta_{i,k}$. If $\mathrm A$ and $\mathrm B$ are covariant tensors, we will also abbreviate the notations by using
\[
\mathrm A:\mathrm B:=\sum_{i,j,k,l} g^{i,j}\,\mathrm A_{j,k}g^{k,l}\,\mathrm B_{l,i}\quad\mbox{and}\quad\|\mathrm A\|^2:=\mathrm A:\mathrm A\,.
\]
It will be convenient to introduce the \emph{trace free Hessian}
\[
\mathrm L\fv:=\mathrm H\fv-\frac1{d-1}\,(\Delta\fv)\,g\,,
\]
Let us define the tensor $\mathrm Z\fv$ and its trace free counterpart by
\[
\mathrm Z\fv:=\frac{\nabla\fv\otimes\nabla\fv}\fv\quad\mbox{and}\quad\mathrm M\fv:=\mathrm Z\fv-\frac1{d-1}\,\frac{|\nabla\fv|^2}\fv\,g\,.
\]
We use the notations $\lambda_\theta$, $\lambda_\star=\lambda_{\theta_\star}$ and $\delta=\frac1{d-1}-\frac1{n-1}$ of the introduction.
\begin{lem}\label{Lem:PoincareGen} Assume that $d\ge3$ and $n>d$. If $\mathsf p$ is a positive function in $C^3(\M)$, then
\[
\iM{\kk_{\M}[\mathsf p]\,\mathsf p^{1-n}}\ge\big[\lambda_\star-(n-2)\,\alpha^2\big]\iM{|\nabla\mathsf p|^2\,\mathsf p^{1-n}}\,.
\]
If $\M=\S^{d-1}$, there is a positive constant $\zeta_\star$ such that
\begin{multline*}
\isph{\kk_{\M}[\mathsf p]\,\mathsf p^{1-n}}\ge\big[\lambda_\star-(n-2)\,\alpha^2\big]\isph{|\nabla\mathsf p|^2\,\mathsf p^{1-n}}\\
+\zeta_\star\,(n-d)\isph{|\nabla\mathsf p|^4\,\mathsf p^{1-n}}\,.
\end{multline*}
\end{lem}
\begin{proof} The Bochner-Lichnerowicz-Weitzenb\"ock formula
\[
\tfrac12\,\Lap\,(|\nabla\fv|^2)=\|\mathrm H\fv\|^2+\nabla(\Lap\fv)\cdot\nabla\fv+\mathrm{Ric}(\nabla\fv, \nabla\fv)
\]
yields that
\begin{multline*}
\mathcal A:=\iM{\(\tfrac12\,\Delta(|\nabla\mathsf p|^2)-\nabla(\Delta\mathsf p)\cdot\nabla\mathsf p-\tfrac1{n-1}\,(\Delta\mathsf p)^2\)\,\mathsf p^{1-n}}\\
=\iM{ \(\|\mathrm H\mathsf p\|^2+\mathrm{Ric}(\nabla\mathsf p,\nabla\mathsf p)-\tfrac1{n-1}\,(\Delta\mathsf p)^2\)\,\mathsf p^{1-n}}\,.
\end{multline*}
Here $\mathrm{Ric}(\nabla\mathsf p, \nabla\mathsf p)$ is the Ricci curvature tensor contracted with $\nabla\mathsf p \otimes \nabla\mathsf p$.

Set $\mathsf p=f^\beta$, where $\beta=\frac2{3-n}$. A straightforward computation shows that
\[
\mathrm H f^\beta=\beta\,f^{\beta-1}\,\big(\mathrm H f+(\beta-1)\,\mathrm Z f\big)
\]
and hence
\begin{multline*}
\iM{ \(\|\mathrm H\mathsf p\|^2-\tfrac1{n-1}\,(\Delta\mathsf p)^2\)\,\mathsf p^{1-n}}\\
=\beta^2 \iM{ \(\|\mathrm H f+(\beta-1)\,\mathrm Z f\|^2-\tfrac1{n-1}\,\big(\mathrm{Tr}\,(\mathrm H f+(\beta-1)\,\mathrm Z f)\big)^2\)}\\
=\beta^2 \iM{ \(\|\mathrm L f+(\beta-1)\,\mathrm M f\|^2+\delta\,\big(\mathrm{Tr}\,(\mathrm H f+(\beta-1)\,\mathrm Z f)\big)^2\)}\,.
\end{multline*}
Next we observe that
\begin{multline*}
\iM{\big(\mathrm{Tr}\,(\mathrm H f+(\beta-1)\,\mathrm Z f)\big)^2}\\
=\iM{\((\Delta\,f)^2+2\,(\beta-1)\,\Delta\,f\,\frac{|\nabla f|^2}{f}+(\beta-1)^2\,\frac{|\nabla f|^4}{f^2}\)}\,.
\end{multline*}
We recall that $d\ge3$ and hence
\[
\frac{|\nabla f|^4}{f^2}=\|\mathrm Z\fv\|^2=\frac{d-1}{d-2}\,\|\mathrm M\fv\|^2\,.
\]
Using integration by parts, we get that
\begin{multline*}
\iM{\Delta\,f\,\frac{|\nabla f|^2}{f}}=\iM{ \frac{|\nabla f|^4}{f^2} }-2 \iM{\mathrm H f:\mathrm Z f}\\
=\frac{d-1}{d-2}\,\|\mathrm M\fv\|^2-2 \iM{\mathrm L f:\mathrm Z f}-\frac{2}{d-1} \iM{ \Delta\,f\,\frac{|\nabla f|^2}f}\,.
\end{multline*}
This yields
\[
\iM{\Delta\,f\,\frac{|\nabla f|^2}{f}}=\frac{d-1}{d+1}\left[\iM{\frac{d-1}{d-2}\,\|\mathrm M\fv\|^2}-2 \iM{\mathrm L f:\mathrm M f} \right]
\]
by noting that one can replace $\mathrm Z f$ by $\mathrm M f$ because $\mathrm L f$ is trace free.

Using the Bochner-Lichnerowicz-Weitzenb\"ock formula once more
we obtain
\[
\iM{(\Delta\,f)^2}=\frac{d-1}{d-2} \iM{\|\mathrm L f\|^2}+\frac{d-1}{d-2} \iM{ \mathrm{Ric}(\nabla f, \nabla f)}\,.
\]
Hence we find that for any $\theta$,
\begin{multline*}
\iM{\big(\mathrm{Tr}\,(\mathrm H f+(\beta-1)\,\mathrm Z f)\big)^2}-(1-\theta)\iM{(\Delta\,f)^2}\\
=\theta\,\left[\frac{d-1}{d-2} \iM{\|\mathrm L f\|^2}+\frac{d-1}{d-2} \iM{ \mathrm{Ric}(\nabla f, \nabla f)}\right]\hspace*{2cm}\\
+2\,(\beta-1)\,\frac{d-1}{d+1}\left[\iM{\frac{d-1}{d-2}\,\|\mathrm M\fv\|^2}-2 \iM{\mathrm L f:\mathrm M f} \right]\\
+(\beta-1)^2\,\frac{d-1}{d-2}\iM{\|\mathrm M\fv\|^2}\,.
\end{multline*}
Altogether, we get
\begin{multline*}
\mathcal A-\beta^2\delta\,(1-\theta)\iM{(\Delta\,f)^2}=\beta^2\iM{\Big(\mathsf a\,\|\mathrm L f\|^2+\,2\,\mathsf b\,\mathrm L f:\mathrm M f+\,\mathsf c\,\|\mathrm M f\|^2\Big)}\\
+\beta^2\(1+\delta\,\theta\,\frac{d-1}{d-2}\)\iM{ \mathrm{Ric}(\nabla f, \nabla f)}\,,
\end{multline*}
where
\begin{eqnarray*}
&&\mathsf a=1+\delta\,\theta\,\frac{d-1}{d-2}\,,\\
&&\mathsf b=(\beta-1)\,\(1-\,2\,\delta\,\frac{d-1}{d+1}\)\,,\\
&&\mathsf c=(\beta-1)^2\,\(1+\delta\,\frac{d-1}{d-2}\)+2\,(\beta-1)\,\frac{\delta\,(d-1)^2}{(d+1)\,(d-2)}\,.
\end{eqnarray*}
The smallest value of $\theta$ for which $\Big(\mathsf a\,\|\mathrm L f\|^2+\,2\,\mathsf b\,\mathrm L f:\mathrm M f+\,\mathsf c\,\|\mathrm M f\|^2\Big)$ is nonnegative is determined by the condition $\mathsf b^2-\mathsf a\,\mathsf c=0$, that is, $\theta=\theta_\star<0$. Notice that with this choice $\theta>-\frac{(d-2)}{\delta\,(d-1)}=-\,\frac{(d-2)\,(n-1)}{n-d}$ for any $d> 2$ and $n>d$, so that the coefficient $\mathsf a$ is always positive.

The conclusion holds by the Poincar\'e inequality
\[
\iM{(\Delta\,f)^2}\ge\lambda_1^{\M}\iM{|\nabla f|^2}\,.
\]
To bound the term involving the Ricci tensor, we simply use the pointwise estimate
\[
\mathrm{Ric}(\nabla\mathsf p,\nabla\mathsf p)\ge\kappa\,|\nabla\mathsf p|^2
\]
and recall that $\frac{d-1}{d-2}\,\kappa\le\lambda_1^{\M}$, with equality when $\M=\S^{d-1}$. Altogether, the function $\theta\mapsto\lambda_\theta:=\big(1+\delta\,\theta\,\frac{d-1}{d-2}\big)\,\kappa+\delta\,(1-\theta)\,\lambda_1^{\M}$ is constant if $\M=\S^{d-1}$, monotone non-increasing otherwise, and we get that
\[
\mathcal A-\lambda_\theta\,\beta^2\iM{|\nabla f|^2}\ge\mathsf a\,\beta^2\iM{\left\|\mathrm L f+\tfrac{\mathsf b}{\mathsf a}\;\mathrm M f\right\|^2}
-\tfrac{\mathsf b^2-\mathsf a\,\mathsf c}{\mathsf a}\,\beta^2\iM{\|\mathrm M f\|^2}\,.
\]
In the general cas, with $d\ge3$, the conclusion holds with $\theta=\theta_\star$, $\mathsf b^2-\mathsf a\,\mathsf c=0$ and $\lambda_\star=\lambda_{\theta_\star}$. If $\M=\S^{d-1}$ with $d\ge3$, we choose $\theta=0$ and get the conclusion.
\end{proof}

\begin{rem}\label{Rem:lambda} Notice that the constant $\lambda_\star$ is an estimate of the largest constant $\lambda$ such that
\[
\iM {\(\tfrac12\,\Delta(|\nabla\mathsf p|^2)-\nabla(\Delta\mathsf p)\cdot\nabla\mathsf p-\tfrac1{n-1}\,(\Delta\mathsf p)^2-\lambda\,|\nabla\mathsf p|^2\)\,\mathsf p^{1-n}}\ge0\,,
\]
for any positive function $\mathsf p\in C^3(\M)$. It is estimated by $\lambda_\theta$ with $\theta\in[\theta_\star,1]$. In the case of the sphere, that is, $\mathfrak M=\S^{d-1}$, we have that $\frac{d-1}{d-2}\,\kappa=\lambda_1^{\M}$ and $\lambda_\theta=\big(1+\delta\,\frac{d-1}{d-2}\big)\,\kappa=\big(\frac{d-2}{d-1}+\delta\big)\,\lambda_1^{\M}$ is independent of $\theta$. Otherwise, by Lichnerowicz' theorem, we know that $\frac{d-1}{d-2}\,\kappa\le\lambda_1^{\M}$ (with strict inequality if $\M\neq\S^{d-1}$, thanks to Obata's theorem). Hence $\theta\mapsto\lambda_\theta$ is a non-increasing function, and since~$\theta_\star$ is always negative, we have a simple lower bound for $\lambda_\star$:
\[
\lambda_\star\ge\lambda_0=\kappa+\delta\,\lambda_1^{\M}\,.
\]
As in~\cite{MR3229793}, a better, nonlocal, estimate is obtained by refining $\lambda_\star$ as
\[
\lambda_\star:=\kern-4pt\inf_{\begin{array}{c}f\in C^3(\M)\\\mbox{s.t.}\,\nabla f\not\equiv0\end{array}}\kern-6pt\frac{\delta\,(1-\theta)\,\iM{(\Delta\,f)^2}+\(1+\delta\,\theta\,\frac{d-1}{d-2}\)\iM{\mathrm{Ric}_g(\nabla f,\nabla f)}}{\iM{|\nabla f|^2}}\,.
\]
\end{rem}

\begin{rem}\label{Rem:abc} With $\theta=\theta_\star$, the constants $\mathsf a$, $\mathsf b$ and $\mathsf c$ are explicit and given by
\[
\mathsf a=\frac{(d-1)\,(d-2)\,(n+1)^2}{(d+1)\,\big[d\,(n^2-n-4)-(n^2-3\,n-2)\big]}\,,\quad\mathsf b=-\frac{(n+1)\,(d-1)}{(n-3)\,(d+1)}\,,
\]
\[
\mathsf c=\frac{(d-1)\,\big[d\,(n^2-n-4)-(n^2-3\,n-2)\big]}{(d-2)\,(d+1)\,(n-3)^2}\,.
\]
If $\M=\S^{d-1}$, we have $\lambda_1^{\M}=d-1$, $\kappa=d-2$ and
\[
\lambda_\star=\frac{n-2}{n-1}\,(d-1)\,,\quad\zeta_\star=-\frac{(d-1)\,\big(d\,(3\,n+5)-3\,n-1\big)}{(d-2)\,(d+1)^2\,(n-3)^2}\,.
\]\end{rem}

\subsection{A Poincar\'e inequality if \texorpdfstring{$d=2$}{d=2}}

The manifold $\M$ is one-dimensional if $d=2$, \emph{i.e.}, it is a smooth closed curve with curvilinear coordinate~$\omega$, of length $1$. A direct computation shows that
\[
\kk_{\M}[\mathsf p]=\frac{n-2}{n-1}\,|\Delta\mathsf p|^2-(n-2)\,\alpha^2\,|\nabla\mathsf p|^2\,.
\]
\begin{lem}\label{Lem:Poincare}
Assume that $\M$ is a smooth closed curve. If $\mathsf p$ is a positive function of class~$C^2(\M)$, then
\begin{multline*}
\iM{\kk_{\M}[\mathsf p]\,\mathsf p^{1-n}}\ge\(\frac{n-2}{n-1}\,\lambda_1^{\M}-(n-2)\,\alpha^2\)\iM{|\nabla\mathsf p|^2\,\mathsf p^{1-n}}\\
+\frac1{12}\,(n+3)\,(n-2)\iM{\frac{|\nabla\mathsf p|^4}{\mathsf p^2}\,\mathsf p^{1-n}}\,.
\end{multline*}\end{lem}
\begin{proof}
Since
\[
\nabla\cdot\(\mathsf p^\frac{1-n}2\,\nabla\mathsf p\)=\mathsf p^\frac{1-n}2\(\Delta\mathsf p+\frac{1-n}2\,\frac{|\nabla\mathsf p|^2}{\mathsf p}\)\,,
\]
we may take the square, integrate by parts the cross term and use the Poincar\'e inequality
\[
\iM{\left|\nabla\cdot\(\mathsf p^\frac{1-n}2\,\nabla\mathsf p\)\right|^2}\ge\lambda_1^{\M}\iM{\left|\mathsf p^\frac{1-n}2\,\nabla\mathsf p\right|^2}\,.
\]
Notice that
\[
\iM{\Delta\mathsf p\,\frac{|\nabla\mathsf p|^2}{\mathsf p}\,\mathsf p^{1-n}}=\frac 13\iM{\nabla\cdot\(|\nabla\mathsf p|^2\,\nabla\mathsf p\)\,\mathsf p^{-n}}=\frac n3\iM{\frac{|\nabla\mathsf p|^4}{\mathsf p^2}\,\mathsf p^{1-n}}
\]
and $-\,\frac1{12}\,(n+3)\,(n-1)=\(\frac{1-n}2\)^2+2\,\frac{1-n}2\,\frac n3$. Hence we get
\[\label{Pincd2}
\iM{|\Delta\mathsf p|^2\,\mathsf p^{1-n}}\ge\lambda_1^{\M}\iM{|\nabla\mathsf p|^2\,\mathsf p^{1-n}}+\frac1{12}\,(n+3)\,(n-1)\iM{\frac{|\nabla\mathsf p|^4}{\mathsf p^2}\,\mathsf p^{1-n}}\,.
\]
The conclusion immediately follows.\end{proof}

\subsection{Consequences for \texorpdfstring{$\mathcal K[\mathsf p]$}{K[p]} and some remarks} With $\Q[\mathsf p]$ defined by~\eqref{QK}, we recall that $\mathcal K[\mathsf p]=\irdmu{\(\Q[\mathsf p]-\frac1n\,(\L\mathsf p)^2\)\mathsf p^{1-n}}$. The following result is a direct consequence of Lemmas~\ref{Lem:Derivmatrixform1}, \ref{Lem:PoincareGen} and~\ref{Lem:Poincare}.
\begin{cor}\label{Cor:Derivmatrixform2} Assume that $d\in\N$, $n\in\R$ and $n>d\ge2$ and consider a function $\mathsf p\in C^3((0,\infty)\times\M)$. Then we have
\begin{multline}\label{Ineq:K}
\mathcal K[\mathsf p]\ge\(1-\frac1n\)\alpha^4\irdmugen{\left|\mathsf p''-\frac{\mathsf p'}r-\frac{\Delta\,\mathsf p}{\alpha^2\,(n-1)\,r^2}\right|^2\,\mathsf p^{1-n}}\\
+2\,\alpha^2\,\irdmugen{\frac1{r^2}\left|\nabla\mathsf p'-\frac{\nabla\mathsf p}r\right|^2\,\mathsf p^{1-n}}\\
+\big[\lambda_\star-(n-2)\,\alpha^2\big]\irdmugen{\frac1{r^4}\,|\nabla\mathsf p|^2\,\mathsf p^{1-n}}\,.
\end{multline}
\end{cor}
\begin{rem}\label{Rem:AdditionalTerms} When $\M=\S^{d-1}$ with $d\ge2$, $\lambda_\star-(n-2)\,\alpha^2=(n-2)\(\alpha_{\rm FS}^2-\alpha^2\)$ with~\hbox{$\alpha_{\rm FS}^2=\frac{d-1}{n-1}$}. The difference of the two terms in~\eqref{Ineq:K} involves an additional term equal to $\zeta_\star\,(n-d)\irdmugen{|\nabla\mathsf p|^4\,\mathsf p^{1-n}}$ if $d\ge3$, where the expression of~$\zeta_\star$ can be found in Remark~\ref{Rem:abc}, but one can choose $\mathsf a=1$ while $\mathsf b$ is unchanged. If $d=2$, the difference of the two terms in~\eqref{Ineq:K} is bounded from below by
\[
\frac1{12}\,(n+3)\,(n-1)\irdmugen{\frac{|\nabla\mathsf p|^4}{\mathsf p^2}\,\mathsf p^{1-n}}\,.
\]\end{rem}
\begin{cor}\label{Cor:DerivFisherSign} Assume that $n> d\ge2$ and that $(n-2)\,\alpha^2\le\lambda_\star$. Then, for any function $\mathsf p\in C^3((0,\infty)\times\M)$, $\mathcal K[\mathsf p]\ge0$ and $\mathcal K[\mathsf p]=0$ if and only if $u=u_\star(t,\cdot)$ for some $t>0$.\end{cor}
\begin{proof} Let us deal first with the case $\M=\S^{d-1}$. In this case the condition $(n-2)\,\alpha^2\le\lambda_\star$ is equivalent to $\alpha\le \alpha_{\rm FS}$. By Lemma~\ref{Lem:PoincareGen}, $\mathsf p$ is only a function of $r$. Moreover, since $\mathcal K[\mathsf p]=0$ we find that $\mathsf p'' = \mathsf p'/r$ for all $r$, which implies that $\mathsf p(r)= a+b\, r^2$ for some constants $a$ and $b$.

Let us now address the case of $\M \not= \S^{d-1}$, which is more delicate. If \hbox{$(n-2)\,\alpha^2\le\lambda_\star$}, the inequality $\mathcal K[\mathsf p]\ge0$ follows from Corollary~\ref{Cor:Derivmatrixform2}. Moreover, under the same assumption,
\begin{multline*}
\mathcal K[\mathsf p]
\ge \(1-\frac1n\)\alpha^4\irdmugen{\left|\mathsf p''-\frac{\mathsf p'}r-\frac{\Delta\,\mathsf p}{\alpha^2\,(n-1)\,r^2}\right|^2\,\mathsf p^{1-n}}\\
+2\,\alpha^2\,\irdmugen{\frac1{r^2}\left|\nabla\mathsf p'-\frac{\nabla\mathsf p}r\right|^2\,\mathsf p^{1-n}}\,.
\end{multline*}
We write
\[
\mathsf p(r,\omega)=\sum_{k\ge0}\alpha_k\,\mathsf p_k(r)\,y_k(\omega)
\]
where $(y_k)_{k\ge0}$ is a basis of eigenfunctions associated with $-\Delta$ and $(\lambda^\M_k)_{k\ge0}$ denotes the corresponding sequence of eigenvalues. Notice that $\lambda^\M_0=0$ and $\lambda^\M_k>0$ for any $k\ge1$. When $\mathcal K[\mathsf p]=0$, then
\begin{multline*}
0=\alpha^4\left(1-\frac1n\right)\sum_{k\ge0}\alpha_k^2\left[\mathsf p_k''-\frac{\mathsf p_k'}r+\frac{\lambda^\M_k\,\mathsf p_k}{\alpha^2\,(n-1)\,r^2}\right]^2\\
+2\,\alpha^2\,\frac1{r^2}\sum_{k\ge0}\alpha_k^2\,\lambda^\M_k\left[\mathsf p_k'-\frac{\mathsf p_k}r\right]^2\,.
\end{multline*}
All the terms in the r.h.s.~are nonnegative, which means that we have to solve simultaneously
\be{eqpk1}
\mathsf p_k''-\frac{\mathsf p_k'}r+\frac{\lambda^\M_k\,\mathsf p_k}{\alpha^2\,(n-1)\,r^2}=0
\ee
for any $k\ge0$ and
\be{eqpk2}
\mathsf p_k'-\frac{\mathsf p_k}r=0
\ee
for any $k\ge1$. The first equation shows that, up to multiplication by an arbitrary non-zero constant,
\[
\mathsf p_k^\pm(r)=r^{\beta_k^\pm}\quad\mbox{with}\quad\beta_k^\pm=1\pm\sqrt{1-\frac{\lambda^\M_k}{\alpha^2\,(n-1)}}\,.
\]
 For $k \ge 1$, equations \eqref{eqpk1} and \eqref{eqpk2} are only compatible if at least one of $\beta_k^\pm=1$, which entails that $\lambda_k^\M = (n-1) \,\alpha^2$. We shall prove that this is never the case for \hbox{$k\ge 1$}. Because $\M \not= \S^{d-1}$ we can use Lichnerowicz' and Obata's theorems to conclude that the strict inequality $\frac{d-1}{d-2}\,\kappa<\lambda_1^{\M}$ holds. This implies that \hbox{$\lambda_*< \frac{n-2}{n-1}\, \lambda^\M_1$}, and hence that $\alpha^2\,(n-1)< \lambda^\M_1$. Altogether, $\mathsf p=\mathsf p_0$ has to be radially symmetric and given by $\mathsf p(r)=a\,+b\,r^2$, for some positive constants $a$ and~$b$. This concludes the proof.\end{proof}
\begin{rem} If $\,\M=\S^{d-1}$, the case $n=d\ge3$ and $\alpha=\alpha_{\,\rm FS}$ corresponds to Sobolev's inequality and the condition $(n-2)\,\alpha^2\le\lambda_\star$ is equivalent to $\alpha^2\le\alpha^2_{\,\rm FS}=\frac{d-1}{n-1}=1$. Our results do not apply to this case, because of course it is well known that there is no rigidity in this case.\end{rem}

\section{Proof of the main results}\label{Sec:Proofs}

Assume that $p\in(2,2^*)$ and consider an optimal function for the Caffarelli-Kohn-Nirenberg inequalities~\eqref{CKN}. Such a solution exists according to \cite{Catrina-Wang-01}. Up to a multiplication by a constant, it solves~\eqref{EL}. Hence Theorem~\ref{Thm:Main1} can be considered as a special case of Theorem~\ref{Thm:Rigidity}. Similarly, we can consider the interpolation inequality~\eqref{ineqcylindergeneral}. For the same reasons as in~\cite{Catrina-Wang-01}, an optimal function exists, which solves~\eqref{Eqn:Cylinder} and the upper bound in Corollary~\ref{Cor:CylinderInequality}, that is, $\Lambda_\star\le\Lambda_{\,\rm FS}=\frac{4\,\lambda_1^{\M}}{p^2-4}$ follows from Proposition~\ref{Prop:FS}.

Corollary~\ref{Cor:CylinderConstant} is equivalent to Theorems~\ref{Thm:Main1} and~\ref{Thm:Rigidity}. The proof of equivalence relies on the Emden-Fowler change of variables~\eqref{EF}. Details are left to the reader. Moreover, it is clear that Corollary~\ref{Cor:CylinderConstant} is a special case of Theorem~\ref{Thm:LV} and Corollary~\ref{Cor:CylinderInequality}, which we prove next.

\medskip Take any positive solution $\varphi\in\mathrm H^1(\mathcal C)$ to~\eqref{Eqn:Cylinder} and recall that by undoing the Emden-Fowler transformation \eqref{EF}, the function $\mathsf p$ defined in \eqref{up} can be written~as
\be{pvarphi}
\mathsf p(r,\omega)=(n-1)\,u^{-\frac1n}(r,\omega)=\tfrac{p+2}{p-2}\,r\,\big(\varphi(-\tfrac{\log r}\alpha,\omega)\big)^{-\frac{p-2}2}\quad\forall\,(r,\omega)\in(0,\infty)\times\M\,,
\ee
with $\alpha=\frac{p-2}2\,\sqrt\Lambda$ and $n=\frac{2\,p}{p-2}$, and it satisfies the equation
\be{eqnp}
\mathsf p\,\mathcal L\,\mathsf p-\frac n2\,|\D\mathsf p|^2=\frac{2\,(n-1)^2}{n-2}\quad\mbox{in}\quad(0,\infty)\times\M\,.
\ee
\begin{lem}\label{Lem:ll1} Let $\alpha\le \alpha_{\rm FS}$. For any positive solution $\mathsf p $ of \eqref{eqnp}, corresponding to $\varphi\in\mathrm H^1(\mathcal C)$,
\[
\irdmugen{\left(\mathsf p\,\mathcal L\,\mathsf p-\frac n2\,|\D\mathsf p|^2-\frac{2\,(n-1)^2}{n-2}\right)(\mathcal L\,u^m)}=-\,n\,(n-1)^{n-1}\,\mathcal K[\mathsf p]\,.
\]\end{lem}
\begin{proof} Take $0<r<R<+\infty$. Then a straightforward integration by parts yields
\begin{multline*}
\int_{(r,R)\times\M}\left(\mathsf p\,\mathcal L\mathsf p-\frac n2\,|\D\mathsf p|^2-\frac{2\,(n-1)^2}{n-2}\right)(\mathcal L\,u^m)\,d\mu\\
=-\,n\,(n-1)^{n-1}\int_{(r,R)\times\M}\left(\frac12\,\L\,|\D\mathsf p|^2-\D\mathsf p\cdot\D\L\mathsf p-\frac1n\,(\L\,\mathsf p)^2\right)\mathsf p^{1-n}\,d\mu\\
+\alpha^2\,r^{n-1}\,\iM{\(\frac n2\,u^m\left(\frac{|\D\mathsf p|^2}{\mathsf p}\right)'+\frac{2\,(n-1)^2}{n-2}\,(u^m)'\)}\Bigg|_r^R\,.
\end{multline*}
The regularity of $\mathsf p$ will be proved in Appendix~\ref{Appendix1}. The boundary term is bounded by a constant times $\mathsf c(r)+c(R)$, where
\be{defc}
\mathsf c(r):=r^{n-1}\iM{\Big(|u'|\,u^{m-1}+u^m\,|\D\mathsf p|\,|\D\mathsf p'|+u^m\,|\D\mathsf p|^2\,\frac{|\mathsf p'|}{\mathsf p}\Big)}\,.
\ee 
By Proposition~\ref{Prop:decayinRd} in the Appendix, $\lim_{r\to 0} \mathsf c(r)=\lim_{R\to \infty}\mathsf c(R)=0$ and this ends the proof.
\end{proof}

\begin{proof}[Proof of Theorem~\ref{Thm:LV}] Let us consider a solution $\varphi$ of \eqref{Eqn:Cylinder}. Define $\mathsf p$ by \eqref{pvarphi}, which then satisfies \eqref{eqnp}. It follows from Lemma~\ref{Lem:ll1} that any positive solution of~\eqref{eqnp} satisfies $\mathcal K[\mathsf p]=0$. By Corollary~\ref{Cor:DerivFisherSign}, whenever $\alpha\le\alpha_{\rm FS}$, we get that $\mathcal K[\mathsf p]=0$ determines the solution and establishes the symmetry result.\end{proof}

\begin{proof}[Proof of Corollary~\ref{Cor:CylinderInequality}] We have to discuss the equality cases. A similar discussion has been done in~\cite[Theorem 4]{MR3229793}. Here we observe that the rigidity result covers the case $\Lambda=\lambda_\star/(p-2)$. Now, if
$\Lambda_\star<\Lambda_{\,\rm FS}$, let us consider $\Lambda_n >\Lambda_\star$ such that $\lim_{n\to+\infty}\Lambda_n=\Lambda_\star$. Then, taking a non-radially symmetric extremal function $\varphi_n$ of \eqref{ineqcylindergeneral} with $\Lambda=\Lambda_n$, by elliptic estimates, we see that the sequence $\{\varphi_n\}_n$ converges uniformly to an extremal solution of \eqref{ineqcylindergeneral} with $\Lambda=\Lambda_\star$. Since for any $\Lambda < \Lambda_{\,\rm FS}$ the radial extremals of \eqref{ineqcylindergeneral} are strict local minima, the radial extremal of \eqref{ineqcylindergeneral} for $\Lambda=\Lambda_\star$ cannot be approached by the sequence $\{\varphi_n\}_n$: see \cite{DELT09} for a similar case. Hence, if $\Lambda_\star<\Lambda_{\,\rm FS}$, at $\Lambda=\Lambda_\star$ there are at least two distinct nonnegative solutions of \eqref{eqlinder}, which contradicts the rigidity property at $\Lambda=\lambda_\star/(p-2)$. Thus we know that $\Lambda_\star>\frac{\lambda_\star}{p-2}$ if $\Lambda_\star<\Lambda_{\,\rm FS}$.\end{proof}

\section{Some consequences}\label{Sec:Consequences}

This section illustrates some consequences of our main results by two further results, respectively on Schr\"odinger operators on cylinders and Hardy type inequalities on the Euclidean space.

\subsection{Spectral estimates for Schr\"odinger operators on cylinders}\label{Sec:Spectral}

Rigidity results and optimality in interpolation inequalities have interesting consequences on spectral estimates for Schr\"odinger operators on cylinders. The results of this section have been announced in \cite{dolbeault:hal-01137403}. Here our goal is to compare
\begin{multline*}
\Lambda(\mu):=\sup\left\{\lambda_1^{\mathcal C}[V]:V\in\mathrm L^q(\mathcal C)\,,\;\nrmcndgen Vq=\mu\right\}\,,\\
\Lambda_\star(\mu):=\Lambda_\R\((\mathrm{vol}_g(\M)\big)^{-1/q}\,\mu\)\\
\mbox{and}\;\Lambda_\R(\mu):=\sup\left\{\lambda_1^{\R}[V]:V\in\mathrm L^q(\R)\,,\;\nrml Vq=\mu\right\}
\end{multline*}
where $-\lambda_1^{\mathcal C}[V]$ and $-\lambda_1^{\R}[V]$ denote the lowest eigenvalues of the Schr\"odinger operators $-\partial^2_s-\,\Delta-V$ and $-\partial^2_s-V$ respectively on $\mathcal C$ and $\R$.

\medskip Assume that $q\in(1,+\infty)$ and let us define
\[
\mu_1:=q\,(q-1)\(\frac{\sqrt\pi\;\Gamma(q)}{\Gamma(q+1/2)}\)^{1/q}\quad\mbox{and}\quad \beta:=\frac{2\,q}{2\,q-1}\,.
\]
Notice that $\mu_1=\mu_\star(\Lambda=1)$ with the notations of Section~\ref{Sec:Intro}. According to \cite{MR0121101,Lieb-Thirring76}, we have
\be{LmabdaStarMu}
\Lambda_\R(\mu)=(q-1)^2\,\(\frac\mu{\mu_1}\)^\beta\quad\forall\,\mu>0\,.
\ee
As a consequence, we obtain the one-dimensional \emph{Keller-Lieb-Thirring inequality}: if $V$ is a nonnegative real valued potential in $\mathrm L^q(\R)$, then we have
\be{ineqn:klt}
\lambda_1^\R[V]\le\Lambda_\R(\nrml Vq)\,.
\ee
Equality holds if and only if, up to scalings, translations and multiplications by a positive constant,
\[
V(s)=\frac{q\,(q-1)}{(\cosh s)^2}=:V_1(s)\quad\forall\,s\in\R
\]
where $\nrml{V_1}q=\mu_1$, $\lambda_1^\R[V_1]=\(q-1\)^2$. Moreover the function $\varphi(s)=(\cosh s)^{1-q}$ generates the corresponding eigenspace. See~\cite{DEL2011} for for more details in the context of Caffarelli-Kohn-Nirenberg inequalities.

The classical \emph{Keller-Lieb-Thirring inequality} in $\R^d$ asserts that for all $\gamma\ge0$ if $d\ge 3$, $\gamma>0$ if $d=2$, and $\gamma>1/2$ if $d=1$, the lowest negative eigenvalue, $-\lambda_1^{\R^d}[V]$, of the operator $-\,\Delta -V$ satisfies
\[
\lambda_1^{\R^d}[V]^\gamma\le\mathrm L^1_{\gamma,d}\,\nrmRd{V_+}{\gamma+d/2}^{\gamma+d/2}\quad\forall\,V\in\mathrm L^q(\R^d)
\]
with optimal constant $\mathrm L^1_{\gamma,d}$. See~\cite{MR0121101,Lieb-Thirring76,Dolbeault2013437} for details.
\begin{prop}\label{Cor:SpectralCylinder} Let $d\ge2$ and $q\in(d/2,+\infty)$. The function $\mu\mapsto\Lambda(\mu)$ is convex, positive and such that, with $\gamma=q-\frac d2$,
\[
\Lambda(\mu)^{q-d/2}\sim\mathrm L^1_{\gamma,\,d}\,\mu^q\quad\mbox{as}\quad\mu\to+\infty\,.
\]
With the notations of Theorem~\ref{Thm:LV}, there exists a positive $\mu_\star$ with
\be{mustar}
\mathrm{vol}_g(\M)^\frac{2}{2q-1}\,\frac{\lambda_\star}{2\,(q-1)}\;\mu_1^\beta\le\mu_\star^\beta\le\mathrm{vol}_g(\M)^\frac{2}{2q-1}\,\frac{\lambda_1^{\mathfrak M}}{(2\,q-1)}\;\mu_1^\beta
\ee
such that
\[
\Lambda(\mu)=\Lambda_\star(\mu)\quad\forall\,\mu\in(0,\mu_\star]\quad\mbox{and}\quad\Lambda(\mu)>\Lambda_\star(\mu)\quad\forall\,\mu>\mu_\star\,.
\]
As a special case, if $\M=\S^{d-1}$, inequalities in~\eqref{mustar} are in fact equalities.
\end{prop}
\begin{proof} The existence of the function $\mu\mapsto\Lambda(\mu)$ is an easy consequence of a H\"older estimate:
\[
\nrmcndgen{\partial_su}2^2+\nrmcndgen{\nabla u}2^2-\iCgen{V\,|u|^2}\ge\nrmcndgen{\partial_su}2^2+\nrmcndgen{\nabla u}2^2-\mu\,\nrmcndgen up^2
\]
with $\mu=\nrmcndgen{V_+}q$ and $q=p/(p-2)$, and of the Gagliardo-Nirenberg inequality~\eqref{ineqcylindergeneral}. Since the equality case in H\"older's inequality is achieved by $V=u^{p-2}$ up to some multiplicative constant, our Keller-Lieb-Thirring inequality
\[
\lambda_1^{\mathcal C}[V]\le\Lambda\big(\nrmcndgen{V_+}q\big)\quad\forall\,V\in\mathrm L^q(\mathcal C)\,.
\]
is in fact exactly equivalent to~\eqref{ineqcylindergeneral} and $\mu\mapsto\Lambda(\mu)$ is the inverse of the function $\Lambda\mapsto\mu(\Lambda)$ in~\eqref{ineqcylindergeneral}. Hence~\eqref{mustar} is equivalent to the estimates of Corollary~\ref{Cor:CylinderInequality}. The estimate of $\Lambda(\mu)$ as $\mu\to+\infty$ and its other properties can be proved exactly as in~\cite{Dolbeault2013437}.
\end{proof}

\subsection{Hardy inequalities with potentials}\label{Sec:Hardy}

With $v(x)=|x|^a\,u(x)$ and $\Lambda=(a_c-a)^2$, the Caffarelli-Kohn-Nirenberg inequalities~\eqref{CKN} can be rewritten as
\[
\ird{|\nabla u|^2}-\(a_c^2-\Lambda\)\ird{\frac{|u|^2}{|x|^2}}\ge\mu(\Lambda)\(\ird{\frac{|u|^p}{|x|^{(b-a)p}}}\)^{2/p}\,.
\]
On the other hand, if $V$ is a given smooth nonnegative potential on $\R^d$ such that $V(0)=0$, then by H\"older's inequality we get that
\[
\ird{V\,\frac{|u|^2}{|x|^2}}=\ird{\frac V{|x|^\frac dq}\,\frac{|u|^2}{|x|^{2\(\frac dp-a_c\)}}}\le\(\ird{\frac{V^q}{|x|^d}}\)^{1/q}\(\ird{\frac{|u|^p}{|x|^{(b-a)p}}}\)^{2/p}.
\]
Let us denote by $\mu\mapsto\Lambda(\mu)$ the inverse of $\Lambda\mapsto\Lambda(\mu)$. Then we have the following result.
\begin{prop}\label{Cor:Hardy} Let $d\ge1$ and $q\in(\min\{1,d/2\},+\infty)$. Assume that $V$ is a nonnegative function such that $|x|^{-d}\,V^q$ is integrable. The for any $u\in\dot{\mathrm H}^1(\R^d)$, we have
\[
\ird{|\nabla u|^2}-\ird{V\,\frac{|u|^2}{|x|^2}}-\(a_c^2-\Lambda(\mu)\)\ird{\frac{|u|^2}{|x|^2}}\ge0\;\mbox{if}\;\mu=\(\ird{\frac{V^q}{|x|^d}}\)^{1/q}.
\]
As a special case, if $\mu=\(\ird{\frac{V^q}{|x|^d}}\)^{1/q}\le\mu_\star$ with $\mu_\star$ defined as in Proposition~\ref{Cor:SpectralCylinder}, then with $\Lambda_\star$ given by~\eqref{LmabdaStarMu}, we have
\[
\ird{|\nabla u|^2}-\ird{V\,\frac{|u|^2}{|x|^2}}-\(a_c^2-\Lambda_\star(\mu)\)\ird{\frac{|u|^2}{|x|^2}}\ge0\quad\forall\,u\in\dot{\mathrm H}^1(\R^d)\,.
\]
\end{prop}
The above result is a generalized form of Hardy's inequality. If $d\ge3$, we recover the usual form by taking $V\equiv0$, with optimal constant $a_c^2$. There is no optimal potential because the equality in H\"older's inequality would mean that $V$ is proportional to $|u|^{p-2}$, so that $|x|^{-d}\,V^q$ is not integrable if $v(x)=|x|^{-a}\,u(x)$ is an optimal function for~\eqref{CKN}.

\appendix\section{Regularity and decay estimates}\label{Appendix1}

We denote by $'$ and $\nabla$ the differentiation with respect to $s$ and $\omega$ respectively. We work in the general setting and do not assume that $\M=\S^{d-1}$.
\begin{prop}\label{Prop:estimates} Any positive solution $\varphi\in\mathrm H^1(\mathcal C)$ of~\eqref{eqlinder} with $p\in (2, 2^*)$ is uniformly bounded and smooth. Moreover there are two positive constants, $C_1$ and $C_2$ such that, for all $(s,\omega)\in\mathcal C$,
\[\label{asympt-beh-ders}
C_1\,e^{-\sqrt\Lambda\,|s|}\le\varphi(s,\omega)\le C_2\,e^{-\sqrt\Lambda\,|s|}\,,
\]
\[\label{asympt-beh-ders2}
|\varphi'(s,\omega)|\kern1.2pt,\;|\varphi''(s,\omega)|\kern1.2pt,\;|\nabla\varphi(s,\omega)|\kern1.2pt,\;|\Delta\,\varphi(s,\omega)|\le C_2\,e^{-\sqrt\Lambda\,|s|}\,.
\]\end{prop}
\begin{proof} A similar result was proved in~\cite{MR1214866}. Here we work in a more general setting when $\M\neq\S^{d-1}$. For sake of completeness, we sketch the main steps of the proof.

\smallskip\noindent\emph{Step 1. The solution is bounded, smooth and $\lim_{|s|\to+\infty}\varphi(s,\omega)=0$ for any \hbox{$\omega\in\M$}.} Boundedness is obtained by a Moser iteration scheme. The $C^\infty$ regularity follows by a localized boot-strap argument based on, \emph{e.g.}, \cite[Corollary~7.11, Theorem~8.10, and Corollary~8.11]{MR1814364}. If $s\mapsto\chi(s)$ is a smooth truncation function such that \hbox{$0\le\chi\le1$}, $\chi\equiv1$ if $|s|\le1$ and $\chi\equiv0$ if $|s|\ge2$, then $\varphi_\varepsilon(s,\omega):=\varphi(s,\omega)\,\big(1-\chi(\varepsilon\,s)\big)$ has an arbitrary small norm in $\mathrm H^1(\mathcal C)$ and $\lim_{\varepsilon\to0_+}\nrmcndgen{\varphi_\varepsilon}\infty=0$, again by a Moser iteration scheme.

\smallskip\noindent\emph{Step 2. Exponential decay of $\varphi$ in $|s|$.} For any $\mu\in(0,\sqrt\Lambda)$, let $h(s):=e^{-\mu\,|s|}$ and define
\[
s_\mu:=\inf\left\{s>0\,:\,|\varphi(\sigma,\omega)|^{p-2}<\Lambda-\mu^2\,,\;\forall\,(\sigma,\omega)\in\mathcal C\cap\{|\sigma|>s\}\right\}\,.
\]
By the Strong Maximum Principle applied to the function $(h-\varphi)$ which solves
\[
-\,\partial^2_s\,(h-\varphi)-\Delta\,(h-\varphi)+\mu^2\,(h-\varphi)\ge\(\Lambda-\mu^2-|\varphi|^{p-2}\)\,\varphi\ge0
\]
for $|s]\ge s_\mu$, we get the estimate
\[
0<\varphi\le\nrmcndgen\varphi\infty\,e^{-\mu\,(|s|-s_\mu)}\quad\forall\,(s,\omega)\in\mathcal C\cap\{|s|>s_\mu\}\,.
\]

\smallskip\noindent\emph{Step 3. Optimal exponential decay of $\varphi$ in $|s|$.} The function $h_1(s,\omega):=e^{-\sqrt\Lambda\,|s|}$ satisfies the equation $-\,\Delta\,h_1+\Lambda\,h_1=0$ on $\mathcal C\cap\{|s|>1\}$. Hence, by the Strong Maximum Principle, we have
\[\label{optestbelow}
\varphi(s,\omega)\ge\(\min_{\mathcal C\cap\{|s|\le 1\}}\varphi\)\,e^{-\sqrt\Lambda\,(|s|-1)}\,.
\]
{}From Step 2 we know that for some positive $M$ and $\bar s$, we have
\[\textstyle
-\,\partial^2_s\,\varphi-\Delta\,\varphi+\(\Lambda-\frac M{s^2}\)\varphi\le 0\quad\mbox{in}\quad\mathcal C\cap\{|s|>\bar s\}\,,
\]
while the function $h_2(s,\omega):=e^{-\sqrt\Lambda\,|s|}\,e^{\frac\lambda{|s|}}$ satisfies
\[\textstyle
-\,\partial^2_s\,h_2-\Delta\,h_2+\(\Lambda-\frac M{s^2}\) h_2=-\frac1{s^2}\(M+2\,\lambda\,\sqrt\Lambda+\frac{2\,\lambda}s+\frac\lambda{s^2}\)h_2\quad\mbox{in}\quad\mathcal C\cap\{|s|>\bar s\}\,.
\]
By taking $\lambda<-\frac M{2\sqrt\Lambda}$ and applying the Strong maximum Principle for $S>0$ large enough, we obtain
\[\label{optestabove}
0<\varphi\le\|\varphi\|_{L^\infty(\mathcal C)}\,e^{-\frac\lambda S}\,e^{-\sqrt\Lambda\,(|s|-S)}\quad\mbox{in}\quad\mathcal C\cap\{|s|>S\}\,.
\]

\smallskip\noindent\emph{Step 4. Optimal exponential decay in $|s|$ for $\nabla\varphi$, $\Delta\,\varphi$.} Using local charts and \cite[Theorem~8.32, p.~210]{MR1814364} on local $C^{1,\alpha}$ estimates, all first derivatives of $\varphi$ converge to~$0$ with rate $e^{-\sqrt\Lambda\,|s|}$ as $|s|\to+\infty$. \cite[Theorem~8.10, p.~186]{MR1814364} provides local $\mathrm W^{k+2,2}$ estimates of the order $e^{-\sqrt\Lambda\,|s|}$ for $|s|$ large enough. The result follows from~\cite[Corollary~7.11, Theorem~8.10, and Corollary~8.11]{MR1814364} if $k$ is taken large enough.\end{proof}

Next we rephrase the results of Proposition~\ref{Prop:estimates} in the language of the \emph{pressure function} $\mathsf p$ of Section~\ref{Sec:Flow} using~\eqref{pvarphi} and establish the estimates needed in Lemmas~\ref{Lem:DerivFisher} and~\ref{Lem:ll1}.
\begin{prop}\label{Prop:decayinRd} Let $m=1-1/n$ and $\varphi\in\mathrm H^1(\mathcal C)$ be a positive solution of~\eqref{Eqn:Cylinder} with $p\in (2, 2^*)$. Then the functions $\mathsf p$ associated with $\varphi$ according to~\eqref{pvarphi} are such that $\mathsf p''$, $\mathsf p'/r$, $\mathsf p/r^2$, $\nabla\mathsf p'/r$, $\nabla\mathsf p/r^2$ and $\Delta\mathsf p/r^2$ are bounded as $r\to+\infty$ and of class $C^\infty$ on $(0,\infty)\times\M$. Moreover, if $\alpha \le \alpha_{\rm FS}$, as $r\to0_+$, we have
\begin{enumerate}
\item[(i)] $\iM{|\mathsf p'(r,\omega)|^2}\le O(1)$,
\item[(ii)] $\iM{|\nabla\mathsf p(r,\omega)|^2}\le O(r^2)$,
\item[(iii)] $\iM{|\mathsf p''(r,\omega)|^2}\le O(1/r^2)$,
\item[(iv)] $\iM{\left|\nabla\mathsf p'(r,\omega)-\tfrac1r\,\nabla\mathsf p(r,\omega)\right|^2}\le O(1)$,
\item[(v)] $\iM{\left|\Delta\mathsf p(r,\omega)\right|^2}\le O(1/r^2)$.
\end{enumerate}
Moreover, with the notations defined by~\eqref{b} and \eqref{defc}, 
\[
\lim_{r\to0_+}\mathsf b(r)=0=\lim_{r\to+\infty}\mathsf b(r)
\]
and
\[
\lim_{r\to0_+}\mathsf c(r)=0=\lim_{r\to+\infty}\mathsf c(r)\,.
\]
\end{prop}
\begin{proof} We say that $f(s,\omega)\sim g(s,\omega)$ as $s\to+\infty$ (resp.~Ê$s\to-\infty$) if the ratio $f/g$ is bounded from above and from below by positive constants, independent of $\omega$, and for $s$ (resp.~$-s$) large enough.

There are some easy consequences of the change of variables~\eqref{pvarphi} and of Proposition~\ref{Prop:estimates}:
since $\varphi(s,\omega)\sim e^{\sqrt\Lambda\,s}$ as $s\to-\infty$, $\varphi(-\log r/\alpha,\omega)\sim r^{-2/(p-2)}$ as $r\to+\infty$ and it is straightforward to check that $\mathsf p''$, $\mathsf p'/r$, $\mathsf p/r^2$, $\nabla\mathsf p'/r$ and $\nabla\mathsf p/r^2$ are bounded as $r\to+\infty$. As a consequence, we obtain that
\[
\mathsf b(r),\; \mathsf c(r)\le O(r^{2-n})\to0\quad\mbox{as}\quad r\to+\infty
\]
because, by assumption, we know that $n>d\ge2$.

To complete the proof, one has to establish that $\lim_{r\to0_+}\mathsf b(r)=\lim_{r\to0_+}\mathsf c(r)=0$. A convenient method for that relies on the Kelvin transformation. Let
\[\label{Kelvinvv}
u(r,\omega)=r^{-2n}\,\widetilde u(R,\omega)\quad\mbox{and}\quad\mathsf p(r,\omega)=r^2\,\widetilde{\mathsf p}(R,\omega)
\]
with $R=1/r$. It is a remarkable fact to observe that $\widetilde u$ solves the same equation as $u$, which can be easily seen after applying the Emden-Fowler transformation $w(r,\omega)=r^{2-n}\,\widetilde w(R,\omega)$ to the function $w$ such that $u(r,\omega)=|w(r,\omega)|^\frac{2\,n}{n-2}$. With evident notations if $\varphi$ and $\widetilde\varphi$ are given in terms of $w$ and $\widetilde w$ by~\eqref{EF}, then $\widetilde\varphi(s,\omega)=\varphi(-s,\omega)$ for any $(s,\omega)\in\R\times\M$ and it is clear that equation~\eqref{Eqn:Cylinder} is invariant under the transformation $s\mapsto-\,s$.

According to Proposition~\ref{Prop:estimates}, $\mathsf p(r,\omega)=r^2\,\widetilde{\mathsf p}(1/r,\omega)$ is bounded away from $0$ and from infinity, and, uniformly in $\omega$,
\begin{eqnarray*}
&&|\mathsf p'(r,\omega)|=|2\,r\,\widetilde{\mathsf p}\(\tfrac1r,\omega\)-\widetilde{\mathsf p}'\(\tfrac1r,\omega\)|\le O\left(\frac1r\(\sqrt\Lambda-\frac{\widetilde\varphi'(s,\omega)}{\widetilde\varphi(s,\omega)}\)\right)\,,\\
&&\tfrac1r\,|\nabla\mathsf p(r,\omega)|=r\,|\nabla\,\widetilde{\mathsf p}\(\tfrac1r,\omega\)|\le O\left(\frac1r\,\frac{\nabla\widetilde\varphi(s,\omega)}{\widetilde\varphi(s,\omega)}\right)\,,\\
\end{eqnarray*}
which are of order at most $1/r$. Moreover, also uniformly in $\omega$,
\begin{eqnarray*}
&&|\mathsf p''(r,\omega)|=|2\,\widetilde{\mathsf p}\(\tfrac1r,\omega\)-\frac2r\,\widetilde{\mathsf p}'\(\tfrac1r,\omega\)+\frac1{r^2}\,\widetilde{\mathsf p}''\(\tfrac1r,\omega\)|\\
&&\hspace*{2cm}\le O\left(\frac1{r^2}\(\frac{\widetilde\varphi''(s,\omega)}{\widetilde\varphi(s,\omega)}-\,\frac p2\,\frac{|\widetilde\varphi'(s,\omega)|^2}{|\widetilde\varphi(s,\omega)|^2}+\,\alpha\,\frac{\widetilde\varphi'(s,\omega)}{\widetilde\varphi(s,\omega)}\)\right)\,,\\
&&|\tfrac1r\,\nabla\mathsf p'(r,\omega)-\tfrac1{r^2}\,\nabla\mathsf p(r,\omega)|=|\nabla\,\widetilde{\mathsf p}\(\tfrac1r,\omega\)-\tfrac1r\,\nabla\,\widetilde{\mathsf p}'\(\tfrac1r,\omega\)|\\
&&\hspace*{2cm}\le O\left(\frac1{r^2}\(\frac p2\,\frac{\widetilde\varphi'(s,\omega)\,\nabla\widetilde\varphi(s,\omega)}{|\widetilde\varphi(s,\omega)|^2}-\frac{\nabla\widetilde\varphi'(s,\omega)}{\widetilde\varphi(s,\omega)}\)\right) \,,\\
&&\frac1{r^2}\,|\Delta\,\mathsf p(r,\omega)|=|\Delta\,\widetilde{\mathsf p}\(\tfrac1r,\omega\)|\le O\left(\frac1{r^2}\(\frac{\Delta\widetilde\varphi(s,\omega)}{\widetilde\varphi(s,\omega)}-\,\frac p2\,\frac{|\nabla\widetilde\varphi(s,\omega)|^2}{|\widetilde\varphi(s,\omega)|^2}\)\right)\,, \\
\end{eqnarray*}
which are of order at most $1/r^2$. This shows that $|\mathsf b(r)|$, $|\mathsf c(r)|\le O(r^{n-4})$ and concludes the proof if $4\le d<n$. When $d=2$ or $3$ and $p>4$, \emph{i.e.}, $n<4$, more detailed estimates are needed. We will actually prove Properties (i)--(v) as $r\to0_+$. Using the fact that $\widetilde\varphi$ and $\varphi$ solve the same equation, this amounts to prove that\begin{enumerate}
\item[(i)] $\iM{\left|\frac{\varphi'(s,\omega)}{\varphi(s,\omega)}-\sqrt\Lambda\right|^2}\le O(e^{2\alpha s})$,
\item[(ii)] $\iM{\left|\frac{\nabla\varphi(s,\omega)}{\varphi(s,\omega)}\right|^2}\le O(e^{2\alpha s})$,
\item[(iii)] $\iM{\left|\frac{\varphi''(s,\omega)}{\varphi(s,\omega)}-\,\frac p2\,\frac{|\varphi'(s,\omega)|^2}{|\varphi(s,\omega)|^2}+\,\alpha\,\frac{\varphi'(s,\omega)}{\varphi(s,\omega)}\right|^2}\le O(e^{2\alpha s})$,
\item[(iv)] $\iM{\left|\frac p2\,\frac{\varphi'(s,\omega)\,\nabla\varphi(s,\omega)}{|\varphi(s,\omega)|^2}-\frac{\nabla\varphi'(s,\omega)}{\varphi(s,\omega)}\right|^2}\le O(e^{2\alpha s})$,
\item[(v)] $\iM{\left|\frac{\Delta\varphi(s,\omega)}{\varphi(s,\omega)}-\,\frac p2\,\frac{|\nabla\varphi(s,\omega)|^2}{|\varphi(s,\omega)|^2}\right|^2}\le O(e^{2\alpha s})$,
\end{enumerate}
as $s\to-\infty$.

\medskip\noindent\emph{Proof of\/} (i). Let us consider a positive solution $\varphi$ to~\eqref{Eqn:Cylinder} and define on $\R$ the function
\[
\varphi_0(s)=\iM{\varphi(s,\omega)}\,.
\]
By integrating~\eqref{Eqn:Cylinder} on $\M$, we know that $\varphi_0$ solves
\[
-\,\varphi_0''+\Lambda\,\varphi_0=\iM{\varphi^{p-1}}=:h_0(s)\sim e^{-(p-1)\sqrt\Lambda\,|s|}\quad\mbox{in}\quad\R\,.
\]
From the integral representation
\[
\varphi_0(s)=\frac{e^{-\sqrt{\Lambda}s}}{2\,\sqrt{\Lambda}}\int_{-\infty}^se^{\sqrt{\Lambda}t}\,h_0(t)\,dt+\frac{e^{\sqrt{\Lambda}s}}{2\,\sqrt{\Lambda}}\int_s^\infty e^{-\sqrt{\Lambda}t}\,h_0(t)\,dt\,,
\]
we deduce that $\varphi_0(s)\sim e^{\sqrt{\Lambda}s}\sim\varphi(s,\omega)$ as $s\to-\infty$ and
\[
\frac{\varphi_0'(s)-\sqrt\Lambda\,\varphi_0(s)}{\varphi(s,\omega)}\sim -\,e^{-2\sqrt{\Lambda}s}\int_{-\infty}^se^{\sqrt{\Lambda}t}\,h_0(t)\,dt=O(e^{2\alpha s})\quad\mbox{as}\quad s\to-\infty\,.
\]
If we define $\psi(s,\omega):=e^{\sqrt\Lambda\,|s|}\,\big(\varphi(s,\omega)-\varphi_0(s)\big)$, we may observe that it is bounded and solves the equation
\be{Eqn:psi}
-\,\partial^2_s\psi-\,\Delta\,\psi-\,2\,\sqrt\Lambda\,\partial_s\psi=e^{\sqrt\Lambda\,|s|}\,\big(\varphi^{p-1}-\varphi_0^{p-1}\big)=:H\le O(e^{-2\alpha|s|})
\ee
and
\[
\frac{\partial_s\varphi(s,\omega)}{\varphi(s,\omega)}-\sqrt\Lambda=O(e^{2\alpha s})+\frac{\partial_s\psi(s,\omega)}{e^{-\sqrt\Lambda\,s}\,\varphi(s,\omega)}\quad\mbox{as}\quad s\to-\infty\,.
\]
We recall that $e^{-\sqrt\Lambda\,s}\,\varphi(s,\omega)$ is bounded from above and from below by positive constants as $s\to-\infty$, and $|e^{-\sqrt\Lambda\,s}\,\partial_s\varphi(s,\omega)|$ is bounded above. As a consequence, we know that $\partial_s H=O(e^{2\alpha s})$ as $s\to-\infty$. Hence we know that
\[
\Big|\frac{\partial_s\varphi(s,\omega)}{\varphi(s,\omega)}-\sqrt\Lambda\Big|\le C\,|\partial_s\psi(s,\omega)|+O(e^{2\alpha s})\,,
\]
where $C$ is a constant.
\medskip We differentiate~\eqref{Eqn:psi} with respect to $s$. The function $\partial_s\psi$ solves
\be{Eqn:psis}
-\,\partial^2_s(\partial_s\psi)-\,\Delta\,(\partial_s\psi)-\,2\,\sqrt\Lambda\,\partial_s(\partial_s\psi)=\partial_sH\,,
\ee
with
\[
|\partial_sH(s,\omega)|\le O(e^{2\alpha s})\quad\mbox{as}\quad s\to-\infty\,.
\]
Let us define
\[
\chi_1(s):=\frac12\iM{|\partial_s\psi|^2}\,,
\]
multiply~\eqref{Eqn:psis} by $\partial_s\psi$ and integrate on $\M$. Using
\[
\chi_1'=\iM{\partial_s\psi\,\partial_s^2\psi}
\]
and
\[
\chi_1''=\iM{\partial_s\psi\,\partial^2_s(\partial_s\psi)}+\iM{|\partial_s^2\psi|^2}\,,
\]
we obtain that the nonnegative function $\chi_1$ solves
\be{eq:chi}
-\,\chi_1''+\iM{|\partial_s^2\psi|^2}+\underbrace{\iM{\(|\nabla(\partial_s\psi)|^2-\lambda_1\,|\partial_s\psi|^2\)}}_{\ge0}+\,2\,\lambda_1\,\chi_1-\,2\,\sqrt\Lambda\,\chi_1'=h_1\,,
\ee
where the Poincar\'e inequality
\[
\iM{|\nabla(\partial_s\psi)|^2}\ge\lambda_1\iM{|\partial_s\psi|^2}
\]
holds because $\iM{\partial_s\psi}=0$ for any $s\in\R$, by definition of $\psi$, and where
\[
h_1:=\iM{\partial_sH\,\partial_s\psi}\,.
\]
{}From the Cauchy-Schwarz inequality, we deduce that
\begin{multline*}
|\chi_1'(s)|^2=\(\iM{\partial_s\psi\,\partial_s^2\psi}\)^2\\
\le\iM{|\partial_s\psi|^2}\iM{|\partial_s^2\psi|^2}=2\,\chi_1(s)\iM{|\partial_s^2\psi|^2}\,,
\end{multline*}
that is
\[
\iM{|\partial_s\psi|^2}\ge\frac{|\chi_1'|^2}{2\,\chi_1}\,,
\]
and reinject this estimate in~\eqref{eq:chi} so that
\[
-\,\chi_1''+\frac{|\chi_1'|^2}{2\,\chi_1}+\,2\,\lambda_1\,\chi_1-\,2\,\sqrt\Lambda\,\chi_1'\le h_1\,.
\]
Let $\zeta_1=\sqrt{\chi_1}$ and observe that it solves
\[
-\,\zeta_1''+\,\lambda_1\,\zeta_1-\,2\,\sqrt\Lambda\,\zeta_1'\le\frac{h_1}{2\,\zeta_1}\,.
\]
By Cauchy-Schwarz inequality, for $s\to -\infty$,\[|h_1 (s)| \le \sqrt{2}\, \(\iM{|\partial_sH|^2}\)^{1/2}\!\!\zeta_1(s)\le C\, e^{2\alpha s}\, \zeta_1(s)\,.
\] Using once more an integral representation of the solution, with $\mu:=\sqrt{\Lambda+\lambda_1}$, it is easy to check that
\[
e^{\sqrt{\Lambda} s}\,\zeta_1(s)\le\frac{e^{-\mu s}}{4\,\mu}\int_{-\infty}^se^{(\mu+\sqrt{\Lambda})t}\,\frac{h_1(t)}{\zeta_1(t)}\,dt+\frac{e^{\mu s}}{4\,\mu}\int_s^\infty e^{(\sqrt{\Lambda}-\mu)t}\,\frac{h_1(t)}{\zeta_1(t)}\,dt \,,
\]
which is enough to deduce that $\zeta_1(s)\le O\big(e^{(\mu - \sqrt\Lambda) s}\big)$ as $s\to-\infty$. Note that the condition that
\[
\mu - \sqrt\Lambda= \sqrt{\Lambda+\lambda_1} - \sqrt\Lambda\ge \alpha
\]
is equivalent to the inequality $\alpha \le \alpha_{\rm FS}$.
Hence we have shown that for $\alpha \le \alpha_{\rm FS}$,
\be{estchi1}
\chi_1(s)\le O\(e^{2\alpha s}\)\quad\mbox{as}\quad s\to-\infty\,.
\ee
This ends the proof of (i).

\medskip\noindent\emph{Proof of\/} (ii). By differentiating~\eqref{Eqn:Cylinder} with respect to $\omega$, we obtain
\[
-\,\partial^2_s\,\nabla\varphi-\,\nabla\,\Delta\,\varphi+\Lambda\,\nabla\varphi=(p-1)\,\varphi^{p-2}\,\nabla\varphi\quad\mbox{in}\quad\mathcal C\,.
\]
We proceed as in case (i). With similar notations, by defining
\[
\chi_2(s):=\frac12\iM{|\nabla\varphi|^2}\,,
\]
after multiplying the equation by $\nabla\varphi$ and using the fact that
\[
\iM{(\Delta\varphi)^2}\ge\lambda_1\iM{|\nabla\varphi|^2}
\]
as, \emph{e.g.}, in \cite[Lemma~7]{MR3229793} and a Cauchy-Schwarz inequality, we obtain
\[
-\,\chi_2''+\frac{|\chi_2'|^2}{2\,\chi_2}+\,2\,(\Lambda+\lambda_1)\,\chi_2\le h_2 \,
\]
with $h_2 := (p-1)\,\iM{\varphi^{p-2} |\nabla \varphi|^2} = O\big(e^{-\,p\,\sqrt\Lambda\,|s|}\big)$. The function $\zeta_2=\sqrt{\chi_2}$ satisfies
\[
-\,\zeta_2''+\,(\Lambda+\lambda_1)\,\zeta_2\le\frac{h_2}{2\,\zeta_2}\,.
\]
By the Cauchy-Schwarz inequality, $h_2/\zeta_2 = O\big(e^{-\,(p-1)\,\sqrt\Lambda\,|s|}\big)$. We easily deduce that
\[
\chi_2(s)\le O\big(e^{2\sqrt{\Lambda+\lambda_1}s}\big)\quad\mbox{as}\quad s\to-\infty\,.
\]
Finally, we observe that $\sqrt{\Lambda+\lambda_1}-\sqrt\Lambda\ge\alpha$ for any $\Lambda\in(0,\Lambda_{\,\rm FS})$ and $\varphi(s,\omega)\sim e^{\sqrt\Lambda\,s}$ as $s\to-\infty$, which ends the proof of (ii).

\medskip\noindent\emph{Proof of\/} (iii). With $\psi=e^{-\sqrt{\Lambda}s}\,(\varphi-\varphi_0)$ and $\varphi_0(s)=\iM{\varphi(s,\omega)}$ as in case (i), we can check that
\begin{multline}
\frac{\varphi''}\varphi-\,\frac p2\,\frac{|\varphi'|^2}{|\varphi|^2}+\,\alpha\,\frac{\varphi'}\varphi= O(e^{2\alpha s})+\frac{e^{\sqrt\Lambda s}}{\varphi}\,\partial_s^2\psi
\\ +\((2-p)\,\sqrt\Lambda +\alpha -p\,\frac{\varphi'_0-\sqrt\Lambda \varphi_0}{\varphi}\)\,\frac{e^{\sqrt\Lambda s}}{\varphi}\,\partial_s\psi - \frac{p}2\, \Big( \frac{e^{\sqrt\Lambda s}}{\varphi} \Big)^2 |\partial_s\psi|^2\,\,.
\end{multline}
Because according to Proposition~\ref{Prop:estimates} $\partial_s\psi$ and $\frac{\partial_s\varphi}\varphi$ are bounded as $s\to -\infty$, and taking into account \eqref{estchi1}, it remains to prove that
\[
\chi_3(s):=\frac12\iM{|\partial_s^2\psi|^2}
\]
is of order $O(e^{2\alpha s})$. We differentiate~\eqref{Eqn:psi} twice with respect to $s$. After multiplying the equation by $\partial_s^2\psi$ and using the fact that
\[
\iM{|\nabla(\partial_s^2\psi)|^2}\ge\lambda_1\iM{|\partial_s^2\psi|^2}
\]
because $\iM{\partial_s^2\psi}=0$, we obtain
\[
-\,\chi_3''+\frac{|\chi_3'|^2}{2\,\chi_3}+\,2\,\lambda_3\,\chi_3-\,2\,\sqrt\Lambda\,\chi_3'\le h_3\,,
\]
with $
h_3:=\iM{\partial^2_sH\,\partial^2_s\psi}\,.
$ With the same arguments as in case (i), we deduce that
\[
\chi_3(s)\le O\big(e^{2(\sqrt{\Lambda+\lambda_1}-\sqrt\Lambda)\,s}\big)\le O(e^{2\alpha s})\quad\mbox{as}\quad s\to-\infty\,.
\]
This ends the proof of (iii).

\medskip\noindent\emph{Proof of\/} (iv). The term $\frac{\varphi'(s,\omega)\,\nabla\varphi(s,\omega)}{|\varphi(s,\omega)|^2}$ is easily bounded after integrating with respect to $\omega$ because $\frac{\partial_s\,\varphi}\varphi$ is bounded according to Proposition~\ref{Prop:estimates} and by (ii). As for the term $\frac{\nabla\varphi'(s,\omega)}{\varphi(s,\omega)}$, we proceed like in case (ii). By applying the operator $\nabla\partial_s$ to~\eqref{Eqn:Cylinder}, we obtain
\begin{multline*}
-\,\partial^2_s\,(\nabla\partial_s\varphi)-\,\nabla\Delta(\partial_s\varphi)+\Lambda\,\nabla\partial_s\varphi=\partial_s\nabla H\\
=(p-1)\,\varphi^{p-2}\(\nabla\partial_s\varphi+(p-2)\,\frac{\partial_s\varphi\,\nabla\varphi}\varphi\)\;\mbox{in}\;\mathcal C\,.
\end{multline*}
With similar notations, by defining
\[
\chi_4(s):=\frac12\iM{|\nabla\partial_s\varphi|^2}\,,
\]
after multiplying the equation by $\nabla\partial_s\varphi$ and using Poincar\'e inequality
\[
\iM{|\Delta(\partial_s\varphi)|^2}\ge\lambda_1\iM{|\nabla(\partial_s\varphi)|^2}
\]
we obtain
\[
-\,\chi_4''+\frac{|\chi_4'|^2}{2\,\chi_4}+\,2\,(\Lambda+\lambda_1)\,\chi_4\le h_4\,,
\]
with $h_4:=\iM{\partial_s \nabla H\,\partial_s \nabla \varphi}$. With the same arguments, we deduce that
\[
\chi_4(s)\le O\big(e^{2(\sqrt{\Lambda+\lambda_1}\,s}\big)\quad\mbox{as}\quad s\to-\infty\,.
\]
We end the proof of (iv) by observing that $\sqrt{\Lambda+\lambda_1}-\sqrt\Lambda\ge\alpha$ for any $\Lambda\in(0,\Lambda_{\,\rm FS})$ and $\varphi(s,\omega)\sim e^{\sqrt\Lambda\,s}$ as $s\to-\infty$.

\medskip\noindent\emph{Proof of\/} (v). By applying the Laplace-Beltrami operator to~\eqref{Eqn:Cylinder}, we obtain
\[
-\,\partial^2_s\,(\Delta\varphi)-\,\Delta^2\varphi+\Lambda\,\Delta\varphi=\Delta H=(p-1)\,\varphi^{p-2}\(\Delta\varphi+(p-2)\,\frac{|\nabla\varphi|^2}\varphi\)\quad\mbox{in}\quad\mathcal C\,.
\]
We proceed as in case (ii). With similar notations, by defining
\[
\chi_5(s):=\frac12\iM{|\Delta\varphi|^2}\,,
\]
after multiplying the equation by $\Delta\varphi$ and using the fact that
\[
-\iM{\Delta\varphi\,\Delta^2\varphi}=\iM{|\nabla\Delta\varphi|^2}\ge\lambda_1\iM{|\Delta\varphi|^2}\,,
\]
we obtain
\[
-\,\chi_5''+\frac{|\chi_5'|^2}{2\,\chi_5}+\,2\,(\Lambda+\lambda_1)\,\chi_5\le h_5
\]
with $h_5:= \iM{\Delta H \, \Delta \varphi}$. With the same arguments, we deduce that
\[
\chi_5(s)\le O\big(e^{2(\sqrt{\Lambda+\lambda_1}\,s}\big)\quad\mbox{as}\quad s\to-\infty\,,\]
and use again the fact that $\sqrt{\Lambda+\lambda_1}-\sqrt\Lambda\ge\alpha$ for any $\Lambda\in(0,\Lambda_{\,\rm FS})$ and $\varphi(s,\omega)\sim e^{\sqrt\Lambda\,s}$ as $s\to-\infty$.
The estimate for the other term follows from (ii). This ends the proof of (v).
\end{proof}

\bigskip\noindent{\bf Acknowledgements.} This work has been partially supported by the projects \emph{STAB} and \emph{Kibord} (J.D.) of the French National Research Agency (ANR). M.L.~has been partially supported by the NSF grant DMS-1301555.
\par\medskip\noindent{\small\copyright\,2015 by the authors. This paper may be reproduced, in its entirety, for non-commercial purposes.}

\bibliographystyle{amsplain}

\begin{thebibliography}{10}

\bibitem{Aubin-76}
Thierry Aubin, \emph{Probl\`emes isop\'erim\'etriques et espaces de {S}obolev},
  J. Differential Geometry \textbf{11} (1976), no.~4, 573--598. \MR{MR0448404
  (56 \#6711)}

\bibitem{MR772092}
Dominique Bakry and Michel {\'E}mery, \emph{Hypercontractivit\'e de
  semi-groupes de diffusion}, C. R. Acad. Sci. Paris S\'er. I Math.
  \textbf{299} (1984), no.~15, 775--778. \MR{MR772092 (86f:60097)}

\bibitem{Bakry-Emery85}
\bysame, \emph{Diffusions hypercontractives}, S\'eminaire de probabilit\'es,
  XIX, 1983/84, Lecture Notes in Math., vol. 1123, Springer, Berlin, 1985,
  pp.~177--206. \MR{88j:60131}

\bibitem{MR3155209}
Dominique Bakry, Ivan Gentil, and Michel Ledoux, \emph{Analysis and geometry of
  {M}arkov diffusion operators}, Grundlehren der Mathematischen Wissenschaften
  [Fundamental Principles of Mathematical Sciences], vol. 348, Springer, Cham,
  2014. \MR{3155209}

\bibitem{MR1412446}
Dominique Bakry and Michel Ledoux, \emph{Sobolev inequalities and {M}yers's
  diameter theorem for an abstract {M}arkov generator}, Duke Math. J.
  \textbf{85} (1996), no.~1, 253--270. \MR{1412446 (97h:53034)}

\bibitem{MR1734159}
Maria~Francesca Betta, Friedemann Brock, Anna Mercaldo, and Maria~Rosaria
  Posteraro, \emph{A weighted isoperimetric inequality and applications to
  symmetrization}, J. Inequal. Appl. \textbf{4} (1999), no.~3, 215--240.
  \MR{1734159 (2001g:35012)}

\bibitem{BV-V}
Marie-Fran{\c{c}}oise Bidaut-V{\'e}ron and Laurent V{\'e}ron, \emph{Nonlinear
  elliptic equations on compact {R}iemannian manifolds and asymptotics of
  {E}mden equations}, Invent. Math. \textbf{106} (1991), no.~3, 489--539.
  \MR{1134481 (93a:35045)}

\bibitem{Caffarelli-Kohn-Nirenberg-84}
Luis Caffarelli, Robert Kohn, and Louis Nirenberg, \emph{First order
  interpolation inequalities with weights}, Compositio Math. \textbf{53}
  (1984), no.~3, 259--275. \MR{MR768824 (86c:46028)}

\bibitem{MR2745814}
Eric~A. Carlen, Jos{\'e}~A. Carrillo, and Michael Loss,
  \emph{Hardy-{L}ittlewood-{S}obolev inequalities via fast diffusion flows},
  Proc. Natl. Acad. Sci. USA \textbf{107} (2010), no.~46, 19696--19701.
  \MR{2745814 (2011k:42032)}

\bibitem{MR1777035}
Jos{\'e}~Antonio Carrillo and Giuseppe Toscani, \emph{Asymptotic {$\mathrm
  L^1$}-decay of solutions of the porous medium equation to self-similarity},
  Indiana Univ. Math. J. \textbf{49} (2000), no.~1, 113--142. \MR{1777035
  (2001j:35155)}

\bibitem{MR1986060}
Jos{\'e}~Antonio Carrillo and Juan~Luis V{\'a}zquez, \emph{Fine asymptotics for
  fast diffusion equations}, Comm. Partial Differential Equations \textbf{28}
  (2003), no.~5-6, 1023--1056. \MR{1986060 (2004a:35118)}

\bibitem{Catrina-Wang-01}
Florin Catrina and Zhi-Qiang Wang, \emph{On the {C}affarelli-{K}ohn-{N}irenberg
  inequalities: sharp constants, existence (and nonexistence), and symmetry of
  extremal functions}, Comm. Pure Appl. Math. \textbf{54} (2001), no.~2,
  229--258. \MR{MR1794994 (2001k:35028)}

\bibitem{MR1214866}
Isabelle Catto and Pierre-Louis Lions, \emph{Binding of atoms and stability of
  molecules in {H}artree and {T}homas-{F}ermi type theories. {III}. {B}inding
  of neutral subsystems}, Comm. Partial Differential Equations \textbf{18}
  (1993), no.~3-4, 381--429. \MR{1214866 (94b:81150c)}

\bibitem{Chou-Chu-93}
Kai~Seng Chou and Chiu~Wing Chu, \emph{On the best constant for a weighted
  {S}obolev-{H}ardy inequality}, J. London Math. Soc. (2) \textbf{48} (1993),
  no.~1, 137--151. \MR{MR1223899 (94h:46052)}

\bibitem{MR1940370}
Manuel Del~Pino and Jean Dolbeault, \emph{Best constants for
  {G}agliardo-{N}irenberg inequalities and applications to nonlinear
  diffusions}, J. Math. Pures Appl. (9) \textbf{81} (2002), no.~9, 847--875.
  \MR{1940370 (2003h:35051)}

\bibitem{MR2381156}
J{\'e}r{\^o}me Demange, \emph{Improved {G}agliardo-{N}irenberg-{S}obolev
  inequalities on manifolds with positive curvature}, J. Funct. Anal.
  \textbf{254} (2008), no.~3, 593--611. \MR{2381156 (2009e:58037)}

\bibitem{Oslo}
Jean Dolbeault and Maria~J. Esteban, \emph{About existence, symmetry and
  symmetry breaking for extremal functions of some interpolation functional
  inequalities}, Abel Symposia (Springer, ed.), 2011, to appear.

\bibitem{Freefem}
\bysame, \emph{A scenario for symmetry breaking in
  {C}affarelli-{K}ohn-{N}irenberg inequalities}, Journal of Numerical
  Mathematics \textbf{20} (2013), no.~3-4, 233---249.

\bibitem{DE2012}
\bysame, \emph{Branches of non-symmetric critical points and symmetry breaking
  in nonlinear elliptic partial differential equations}, Nonlinearity
  \textbf{27} (2014), no.~3, 435.

\bibitem{1406}
Jean Dolbeault, Maria~J. Esteban, Stathis Filippas, and Achiles Tertikas,
  \emph{{Rigidity results with applications to best constants and symmetry of
  Caffarelli-Kohn-Nirenberg and logarithmic Hardy inequalities}}, To appear in
  Caluculus of Variations and PDE, December 2014.

\bibitem{DEKL}
Jean Dolbeault, Maria~J. Esteban, Michal Kowalczyk, and Michael Loss,
  \emph{Improved interpolation inequalities on the sphere}, Discrete and
  Continuous Dynamical Systems Series S (DCDS-S) \textbf{7} (2014), no.~4,
  695--724.

\bibitem{DoEsLa}
Jean Dolbeault, Maria~J. Esteban, and Ari Laptev, \emph{Spectral estimates on
  the sphere}, Analysis \& PDE \textbf{7} (2014), no.~2, 435--460.

\bibitem{Dolbeault2013437}
Jean Dolbeault, Maria~J. Esteban, Ari Laptev, and Michael Loss, \emph{Spectral
  properties of {S}chr{\"o}dinger operators on compact manifolds: Rigidity,
  flows, interpolation and spectral estimates}, Comptes Rendus Mathematique
  \textbf{351} (2013), no.~11--12, 437 -- 440.

\bibitem{DEL2011}
Jean Dolbeault, Maria~J. Esteban, and Michael Loss, \emph{Symmetry of extremals
  of functional inequalities via spectral estimates for linear operators}, J.
  Math. Phys. \textbf{53} (2012), no.~P, 095204.

\bibitem{MR3229793}
\bysame, \emph{Nonlinear flows and rigidity results on compact manifolds}, J.
  Funct. Anal. \textbf{267} (2014), no.~5, 1338--1363. \MR{3229793}

\bibitem{dolbeault:hal-01137403}
\bysame, \emph{{Keller-Lieb-Thirring inequalities for Schr{\"o}dinger operators
  on cylinders}}, Preprint hal-01137403, March 2015.

\bibitem{DELT09}
Jean Dolbeault, Maria~J. Esteban, Michael Loss, and Gabriella Tarantello,
  \emph{On the symmetry of extremals for the {C}affarelli-{K}ohn-{N}irenberg
  inequalities}, Adv. Nonlinear Stud. \textbf{9} (2009), no.~4, 713--726.
  \MR{MR2560127}

\bibitem{MR2437030}
Jean Dolbeault, Maria~J. Esteban, and Gabriella Tarantello, \emph{The role of
  {O}nofri type inequalities in the symmetry properties of extremals for
  {C}affarelli-{K}ohn-{N}irenberg inequalities, in two space dimensions}, Ann.
  Sc. Norm. Super. Pisa Cl. Sci. (5) \textbf{7} (2008), no.~2, 313--341.
  \MR{2437030 (2009g:46059)}

\bibitem{1501}
Jean Dolbeault and Giuseppe Toscani, \emph{Nonlinear diffusions: extremal
  properties of {B}arenblatt profiles, best matching and delays}, Preprint
  hal-01103574, 2015.

\bibitem{Felli-Schneider-03}
Veronica Felli and Matthias Schneider, \emph{Perturbation results of critical
  elliptic equations of {C}affarelli-{K}ohn-{N}irenberg type}, J. Differential
  Equations \textbf{191} (2003), no.~1, 121--142. \MR{MR1973285 (2004c:35124)}

\bibitem{MR615628}
Basilis Gidas and Joel Spruck, \emph{Global and local behavior of positive
  solutions of nonlinear elliptic equations}, Comm. Pure Appl. Math.
  \textbf{34} (1981), no.~4, 525--598. \MR{615628 (83f:35045)}

\bibitem{MR1814364}
David Gilbarg and Neil~S. Trudinger, \emph{Elliptic partial differential
  equations of second order}, Classics in Mathematics, Springer-Verlag, Berlin,
  2001, Reprint of the 1998 edition. \MR{1814364 (2001k:35004)}

\bibitem{Gross75}
Leonard Gross, \emph{Logarithmic {S}obolev inequalities}, Amer. J. Math.
  \textbf{97} (1975), no.~4, 1061--1083. \MR{54 \#8263}

\bibitem{MR797051}
Miguel~A. Herrero and Michel Pierre, \emph{The {C}auchy problem for
  {$u_t=\Delta u^m$} when {$0<m<1$}}, Trans. Amer. Math. Soc. \textbf{291}
  (1985), no.~1, 145--158. \MR{797051 (86i:35065)}

\bibitem{MR1731336}
Toshio Horiuchi, \emph{Best constant in weighted {S}obolev inequality with
  weights being powers of distance from the origin}, J. Inequal. Appl.
  \textbf{1} (1997), no.~3, 275--292. \MR{MR1731336 (2000k:35110)}

\bibitem{MR0121101}
Joseph~B. Keller, \emph{Lower bounds and isoperimetric inequalities for
  eigenvalues of the {S}chr\"odinger equation}, J. Mathematical Phys.
  \textbf{2} (1961), 262--266. \MR{0121101 (22 \#11847)}

\bibitem{MR1338283}
Jean~Ren{\'e} Licois and Laurent V{\'e}ron, \emph{Un th\'eor\`eme d'annulation
  pour des \'equations elliptiques non lin\'eaires sur des vari\'et\'es
  riemanniennes compactes}, C. R. Acad. Sci. Paris S\'er. I Math. \textbf{320}
  (1995), no.~11, 1337--1342. \MR{1338283 (96e:58166)}

\bibitem{MR1631581}
\bysame, \emph{A class of nonlinear conservative elliptic equations in
  cylinders}, Ann. Scuola Norm. Sup. Pisa Cl. Sci. (4) \textbf{26} (1998),
  no.~2, 249--283. \MR{1631581 (99g:35038)}

\bibitem{Lieb-83}
Elliott~H. Lieb, \emph{Sharp constants in the {H}ardy-{L}ittlewood-{S}obolev
  and related inequalities}, Ann. of Math. (2) \textbf{118} (1983), no.~2,
  349--374. \MR{MR717827 (86i:42010)}

\bibitem{Lieb-Thirring76}
Elliott~H. Lieb and Walter~E. Thirring, \emph{Inequalities for the moments of
  the eigenvalues of the schr{\"o}dinger hamiltonian and their relation to
  sobolev inequalities}, pp.~269--303, Essays in Honor of Valentine Bargmann,
  E. Lieb, B. Simon, A. Wightman Eds. Princeton University Press, 1976.

\bibitem{MR2053993}
Chang-Shou Lin and Zhi-Qiang Wang, \emph{Erratum to: ``{S}ymmetry of extremal
  functions for the {C}affarelli-{K}ohn-{N}irenberg inequalities'' [{P}roc.\
  {A}mer.\ {M}ath.\ {S}oc.\ {132} (2004), no.\ 6, 1685--1691]}, Proc. Amer.
  Math. Soc. \textbf{132} (2004), no.~7, 2183. \MR{MR2053993 (2005e:26030)}

\bibitem{Lin-Wang-04}
\bysame, \emph{Symmetry of extremal functions for the
  {C}affarelli-{K}ohn-{N}irenberg inequalities}, Proc. Amer. Math. Soc.
  \textbf{132} (2004), no.~6, 1685--1691. \MR{MR2051129 (2005e:26029)}

\bibitem{MR3200617}
Giuseppe Savar{\'e} and Giuseppe Toscani, \emph{The concavity of {R}\'enyi
  entropy power}, IEEE Trans. Inform. Theory \textbf{60} (2014), no.~5,
  2687--2693. \MR{3200617}

\bibitem{MR2001882}
Didier Smets and Michel Willem, \emph{Partial symmetry and asymptotic behavior
  for some elliptic variational problems}, Calc. Var. Partial Differential
  Equations \textbf{18} (2003), no.~1, 57--75. \MR{MR2001882 (2004m:35092)}

\bibitem{Talenti-76}
Giorgio Talenti, \emph{Best constant in {S}obolev inequality}, Ann. Mat. Pura
  Appl. (4) \textbf{110} (1976), 353--372. \MR{MR0463908 (57 \#3846)}

\bibitem{vazquez2004asymptotic}
Juan~Luis V{\'a}zquez, \emph{Asymptotic behaviour for the porous medium
  equation posed in the whole space}, Nonlinear Evolution Equations and Related
  Topics, Springer, 2004, pp.~67--118.

\end{thebibliography}

\end{document}